\documentclass[twocolumn]{article}
\usepackage[utf8]{inputenc}
\usepackage[margin=1in]{geometry}
\usepackage{amsmath} 
\usepackage{amsthm}

\usepackage[nolist]{acronym}
\usepackage{xcolor}
\usepackage{amsmath}
\usepackage{amssymb}
\usepackage{mathtools}
\usepackage{subcaption}
\usepackage{hyperref}
\hypersetup{colorlinks=true, linkcolor=black}
\usepackage{algorithm}%
\usepackage{algpseudocode}%
\algblock{Input}{EndInput}
\algnotext{EndInput}
\algblock{Output}{EndOutput}
\algnotext{EndOutput}

\usepackage{mathrsfs}
\usepackage{tikz}

\newcommand\tm{t^{-}}
\newcommand\tp{t^{+}}

\newcommand{\rw}[1]{#1}

\newcommand{\N}{\mathbb{N}}

\newcommand{\Dc}{\mathcal{D}}

\newcommand{\s}{\mathcal{S}}

\newcommand{\vs}{\mathcal{V}}
\newcommand{\es}{\mathcal{E}}
\newcommand{\gs}{\mathcal{G}}
\newcommand{\Ns}{\mathcal{N}}

\newcommand{\Us}{\mathcal{U}}

\newcommand{\Rs}{\mathcal{R}}
\newcommand{\Ws}{\mathcal{W}}

\DeclarePairedDelimiter{\abs}{\lvert}{\rvert}
\DeclarePairedDelimiter{\norm}{\lVert}{\rVert}

\DeclarePairedDelimiterX{\inp}[2]{\langle}{\rangle}{#1, #2}

\DeclareMathOperator*{\argmin}{arg\!\,min}

\newtheorem{thm}{Theorem}[section]
\newtheorem{lem}[thm]{Lemma}
\newtheorem{prop}[thm]{Proposition}
\newtheorem{prob}[thm]{Problem}
\newtheorem{cor}[thm]{Corollary}
\newtheorem{defn}[thm]{Definition}

\newtheorem{exmp}[thm]{Example}
\newtheorem{rmk}[thm]{Remark} 

\title{\textbf{Convergence and Robustness Bounds for Distributed Asynchronous  Shortest-Path}}
\makeatletter
\def\thanks#1{\protected@xdef\@thanks{\@thanks \protect\footnotetext{#1}}}
\makeatother
\author{Jared Miller,   Mattia Bianchi,   Florian D\"orfler  
\thanks{
J. Miller is with the Chair of Mathematical Systems Theory, Department of Mathematics,  University of Stuttgart, Stuttgart, Germany (e-mail: jared.miller@imng.uni-stuttgart.de).}
\thanks{M. Bianchi and F. Dörfler are with the Automatic Control Laboratory (IfA), Department of Information Technology and Electrical Engineering (D-ITET), ETH Z\"{u}rich, Physikstrasse 3, 8092, Z\"{u}rich, Switzerland 
(e-mail: mbianch@ethz.ch, doerfler@ethz.ch).}
\thanks{J. Miller   was partially supported by the Swiss National Science Foundation under NCCR Automation, grant agreement 51NF40\_180545, and by the Deutsche Forschungsgemeinschaft (DFG, German Research Foundation) under Germany's Excellence Strategy - EXC 2075 – 390740016. We acknowledge the support by the Stuttgart Center for Simulation Science (SimTech). }}
\date{\today}

\begin{document}

\maketitle

\section{Introduction}
\label{sec:introduction}
Shortest path problems are of fundamental importance  for a variety of routing 
applications, including in communication networks (e.g., internet protocols \cite{bley2010}),  path planning (e.g., for autonomous vehicles \cite{Katona2024}), personal navigation (e.g. Google Maps),   and decision making in dynamic environments (e.g., traffic management and supply chain logistics \cite{Pallottino1998}).
Mathematically, we consider a directed graph $\mathcal{G} = (\mathcal{V},\mathcal{E})$ with vertices $\mathcal V$ and edges $\mathcal{E}$, where each edge $(i,j) \in \mathcal{E}$ is equipped with a positive edge weight (or ``length'') $e_{ij}$. 
The shortest path problem consists of finding a path of minimum length from each vertex $i\in \mathcal{V}$ to a set of targets (or ``sources'') $S \subset \mathcal \vs$. 
Letting $\text{Path}(i, S)$ be the set of paths in $\gs$ originating at $i \in \vs$ and terminating at $S$, the problem can also be formulated in terms of the shortest path distance:
\begin{prob}    \label{prob:shortest}
    Given a graph $\gs$ and edge-weights $e$,   compute the following shortest-path distance for each agent $i \in \vs$:
    \begin{align}\label{eq:shortestdistance}
    d_i^* = &\ \textrm{dist}(i, S)         =  \min_{\mathcal{P} \in \text{Path}(i, S)} \sum_{(v, v') \in \mathcal{P}} e_{v, v'}.
\end{align}
\end{prob}

\medskip
Solution methods for this problem often use dynamic programming approaches, because the shortest path distance $d_i^*$ from node $i$ to the set $S$ satisfies an optimality principle\cite{bellman1958routing}:
\begin{align}
    \forall i &\in S: &\quad d_i^* &= 0. \label{eq:optimality}\\
    \forall i &\in \vs\setminus S: & \quad d^*_i &= \min_{j \in \Ns(i)}{e_{ij} + d^*_j}, \nonumber \end{align} 
where  $\Ns(i)$ is  the set of in-neighbors of $i \in \vs$. Note that a feasible shortest path $\mathcal{P} \in \text{Path}(i, S)$ can be constructed from a solution $\{d_i^*\}_{i\in\vs}$ to the shortest distance Problem~1.\ref{eq:shortestdistance} by considering a predecessor of vertex $i$ (i.e., any vertex in the  $\argmin$ in \eqref{eq:optimality}), then a predecessor's predecessor, etc.

The most common approaches to solve the shortest path distance problem are Dijkstra's method \cite{Dijkstra59} and the Bellman--Ford algorithm \cite{Ford56,Bellman58}. 
Dijkstra's method has lower algorithmic complexity, but unfortunately it is not  suitable for distributed computation.
Instead, the Bellman--Ford algorithm can be implemented in a distributed manner. 
This is crucial in applications where each node is a processor/computer that can contribute to numerical solution, for instance in internet routing problems.
In a distributed framework, each agent $i$ separately computes a distance estimate $d_i$ based only on the neighboring estimates $\{d_{j}\}_{j \in \mathcal{N}(i)}$. In particular, 
the distributed Bellman--Ford algorithm can be written 
as: $\forall t=1,2,\ldots,$
\begin{align}
d_i(t) &= \begin{cases}
        \min_{ j \in \{ i\} \cup \Ns(i)   } d_j(t-1) + e_{ij} & i \not\in S \\
        0 & i \in S
    \end{cases},\label{eq:bellman_ford_orig}
\end{align}
where $e_{ii} = 0$ for all $i\in\vs$. 

The suitability of the Bellman--Ford algorithm for distributed computation has earned it great popularity; we refer to  \cite{Busato2016,GOLDBERG19933,WEI2024123311,Cantone2019} for its generalizations and accelerations, and to \cite{walden2005bellman} for a historical overview.  
The algorithm requires initial overestimates $d_i(0) \geq d_i^*$, since the sequence $\{d_i(t)\}_{t \geq 0}$ is nonincreasing for each node $i \in \vs$. Therefore, the algorithm \eqref{eq:bellman_ford_orig} will  not converge to the true shortest-distances $d^*$ if  any agent has an initial underestimate  ($\exists i \in \vs: \ d_i(0) \leq d^*_i$). 

This non-convergence issue is solved by the Adaptive Bellman--Ford (ABF) algorithm: $\forall t=1,2,\ldots,$
\begin{align}\label{eq:abf} \ d_i(t) = \begin{cases}
        \min_{ j \in  \Ns(i)   } d_j(t-1) + e_{ij} & i \not\in S \\
        0 & i \in S
    \end{cases}. 
\end{align}
In \eqref{eq:abf}, minimization is performed over $j \in \Ns(i)$, rather than over the set $j \in \{i\} \cup \Ns(i)$ as in \eqref{eq:bellman_ford_orig}. The ABF method is \emph{self-stabilizing}  \cite{dijkstra1974self}, namely it converges  to the solution of \eqref{eq:optimality} in a finite number of steps \emph{for any initial condition} \cite[Ch.~4]{Bertsekas1989}.

In this paper, we explore the convergence properties of a distributed and asynchronous execution of the  {ABF} method. 
Asynchronous computation has several key advantages: it eliminates the 
synchronization cost,  improves robustness against
unreliable \rw{and delayed} communication, reduces idle time,  decreases memory-access, and alleviates transmission congestion. 
The convergence  of the asynchronous {ABF}  algorithm was first studied by Bertsekas, motivated by the routing algorithm used in ARPANET \cite{bertsekas1982distributed}. The convergence result was then extended to a more general ``totally asynchronous'' scenario \cite[Ch.~6]{Bertsekas1989}. While both algorithms were shown to be self-stabilizing, no finite-time  bound was provided.  Most recently, finite-time convergence and robustness bounds were proven for the \emph{synchronous} {ABF} by Mo, Dasgupta et al. \cite{mo2021lyapunov,mo2019robustness,dasgupta2016lyapunov}. Their  analysis is based  on finding a suitable ``Lyapunov'' function that monotonically decreases along the iterations. This Lyapunov-type approach has the advantage of guaranteeing  global uniform asymptotic stability of the algorithm, and in turn, a certain measure of  robustness to  perturbations. 
In fact, shortest-path algorithms are frequently used as a building block in complex systems (e.g., aggregate computing applications \cite{Beal2008}). Guaranteeing their robustness and stability is therefore  critical to ensure that they can be safely integrated into larger systems with interacting components \cite{mo2023small}; see \cite{Dorfler_2024} for a recent perspective on interacting algorithmic systems. As in  \cite{mo2021lyapunov}, we are interested in finite convergence and robustness  bounds, but for the \emph{asynchronous} {ABF} algorithm.

\textbf{Contributions:} In this paper, we study convergence and robustness bounds of a distributed and asynchronous implementation of the {ABF} algorithm for shortest path computation. \rw{We mathematically model asynchrony via additional buffer variables, which enable us to account for
 heterogeneous computation/communication delays and update
frequencies, differently from \cite{mo2021lyapunov,mo2019robustness,dasgupta2016lyapunov}}.
Unlike the  ``partially asynchronous'' and ``totally asynchronous'' scenarios in \cite{Bertsekas1989,bertsekas1982distributed}, our asynchrony model does not allow for out-of-order messages, but it accounts for the possibility of race conditions. In turn, this allows us to \rw{produce Lyapunov functions that are monotonically decreasing along updates.}
Our contributions are the following:

\begin{itemize}
    \item We provide  explicit finite-time convergence bounds for the  distributed {ABF} method in the asynchronous setting. Our bounds on the convergence time are simply the bounds from 
    \cite{mo2019robustness} \rw{multiplied by} a factor that depends on the asynchrony measure;

    \item We provide ultimate bounds for the shortest-distance error under persistent perturbations, including bounded communication and structural noise. We   consider both an existing double-sided noise model and a novel single-sided noise model. Furthermore, the ultimate bounds we obtain are independent of the asynchrony measure.

    \item We introduce two new Lyapunov functions for the asynchronous {ABF} method. The Lyapunov functions we propose include  the greatest underestimation and overestimation error as in \cite{mo2021lyapunov}, but they also account for the maximal errors on the buffer variables caused by asynchrony.
    
\end{itemize}

We illustrate our results on several examples, including application to the  Most Probable Path problem \cite{Brraunstein2003}, where our results guarantee bounded finite-time convergence and provable ultimate bounds against multiplicative perturbation, all  in the asynchronous scenario.

The rest of the paper is organized as follows. Section \ref{sec:preliminaries} introduces the  computation model for asynchronous shortest-path computation and assumptions for its execution. Section \ref{sec:convergence} presents finite-time convergence bounds for the asynchronous setting. Section \ref{sec:robustness} derives robustness bounds for the worst-case errors in distance estimates in the presence of noise. 
Section \ref{sec:mpp} adapts our results to the most-probable-path setting. Section \ref{sec:examples} provides  numerical examples to  our convergence and robustness guarantees. Section \ref{sec:conclusion} concludes the paper.

\section{Preliminaries}

\label{sec:preliminaries}

\subsection{Asynchronous Execution}

We next describe an asynchronous computation model for the shortest-path problems. \rw{We first introduce the mathematical model, and then explain how it can be used to describe situations such as communication or computation delays (see Remark~\ref{rmk:asyncmodel} and Figure~\ref{fig:asyncmodel} below).}

In the asynchronous iteration we consider, the nodes do not have a common clock, and do not execute \eqref{eq:abf} at the same time \cite{bertsekas1982distributed}.
\rw{To model this scenario mathematically, a} discrete-time index $t$  is used to count the instants at which one or more agents perform an operation, while an index $t'$ denotes the continuous-time. As an example, if agents perform an action at continuous-times $t' = \{0, 0.4, 1.333, 6, 6.1\}$ seconds, then these times are mapped in increasing order to $t = \{ 0, 1, 2, 3, 4\}$ steps/units.
Each agent $i$ is equipped with an internal distance buffer $d_{i}$, a neighbor outbox buffer $\{m_{ij}\}_{j \in \Ns(i)}$, and a neighbor inbox buffer $\{d_{ij}\}_{j \in \Ns(i)}$ . 
The internal and neighbor buffers can be concatenated together to describe a trajectory:
\begin{align}
    D(t) &= \left[\{d_i(t)\}_{i \in \vs}, \{m_{ij}(t)\}_{(i,j) \in \es},\{d_{ij}(t)\}_{(i, j) \in \es} \right]\!. \label{eq:trajectory}
    \intertext{The optimal value of the distance buffers is}
    D^* &= \left[\{d_i^*\}_{i \in \vs}, \{d_j^*\}_{(i,j) \in \es},\{d_j^*\}_{(i, j) \in \es} \right]\!. \label{eq:steady_state}
\end{align}

Execution of the asynchronous distance computation will be modeled through the following \rw{agent} actions:
\begin{align}
 \intertext{\textbf{\rw{Read}}: agent $i$ acquires a transmission from $j \in \Ns(i)$:}
\text{Read}(i, j): & &  d_{ij} &\leftarrow m_{ij}. \label{eq:transmit}
    \intertext{\textbf{Update}: agent $i$ computes a new internal distance estimate:}
        \text{Update}(i): & & d_{i} &\leftarrow \begin{cases}
    \min_{j \in \Ns(i)} d_{ij} + e_{ij} & i \not\in S \\
    0 & i \in S
\end{cases} \label{eq:update} \\
\intertext{\textbf{Write}: agent $j\in\Ns(i)$ exposes the internal distance computation to agent $i$:}
\text{Write}(i, j): & &  m_{ij} &\leftarrow d_j \label{eq:write}
\end{align}
\rw{The Read, Update, and Write actions can be seen as instances of Input, Internal, and Output actions from
    \cite[Chapter 4.1.2] {alur2015principles}.}
The times at which Read, Update, and Write operations occur will be modeled as a discrete event system \cite{cassandras2008introduction} \rw{involving the Read set $\Rs(t) \subseteq \es , $ the Update set $\Us(t) \subseteq \vs,$ and the Write set $\Ws(t) \subseteq \es$ at each time $t \in \N$. $\Rs(t)$ is the set of edges $(i, j) \in \es$ such that agent $i$ Reads $j$'s outbox at time $t$. $\Us(t)$ is the set of agents $i \in \vs$ that Update their internal distance estimate at time $t$. $\Ws(t)$ is the set of edges $(i, j) \in \es$ such that agent $j$ Writes its outbox buffer $m_{ij}$ at time $t$.  The agent $i$ idles at time $t$ if $i \not\in \mathcal{U}(t) $. Similarly, the edge $(i, j)$ idles if $(i, j) \not\in \mathcal{R}(t)\cup \mathcal{W}(t)$.
}

\label{sec:asynchronous}

Race conditions \rw{within a single time $t \in \N$} will be modeled by an execution queue $\mathcal{Q}(t)$. The execution queue at time $t$ is an ordered list $\mathcal{Q}(t) = \{ \text{INSTR}(k)\}_{k=1}^{\abs{\mathcal{Q}(t)}}$, in which each $\text{INSTR}(k)$ is a Read, Update, or Write instruction. As an example, the execution queue for the action sequence at time step $t=4$ where  vertex 9 Writes to vertex 2, vertex 2 then Reads from vertex 5, and finally vertex 5 Updates is:
\begin{equation}
    \mathcal{Q}(4) = \{\text{Write(2, 9)}, \text{Read(2, 5)},  \text{Update(5)}\}. \nonumber
\end{equation}
The execution queue is bounded as $\forall t: \abs{\mathcal{Q}(t)} \leq  \abs{\vs} + 2\abs{\es}$, as each node can Update  at most once per time step $t$, and each edge can Read and Write at most once per time step.

\begin{algorithm}[!b]
 \caption{Asynchronous Shortest-Path Computation} \label{alg:path_alg_nd}
        \begin{algorithmic}
  \Input $\ $ Graph $\gs(\vs, \es)$, edge weights $e$, source set $S$, Read event stream $\mathcal{R}(\cdot)$, Update event stream $\Us(\cdot),$ Write event stream $\mathcal{W}(\cdot)$, initial distances $d_{i}^0, m_{ij}^0, d_{ij}^0$, Execution queue stream $\mathcal{Q}(\cdot)$
  \EndInput
  \Output $\ $ Distance estimate streams $\{d_i(\cdot)\}_{i \in \vs},$ $\{m_{ij}(\cdot)\}_{(j,i) \in \es},$ and $\{d_{ij}(\cdot)\}_{(i, j) \in \es}$
  \EndOutput
\State Initialize time $t \leftarrow 0$, shortest-path estimates  $\{d_i(0)\leftarrow d_i^0\}_{i \in \vs},$ $\{m_{ij}(0)\leftarrow m_{ij}^0\}_{(i,j) \in \es},$ and $\{d_{ij}(0) \leftarrow d_{ij}^0 \}_{(i, j) \in \es}$
\Loop 
    \State Retrieve current execution queue $\mathcal{Q}(t) \neq \varnothing$
    \For{$k \in \{ 1,\ldots,\abs{\mathcal{Q}(t)}$: \} }  
\If{INSTR$(k)$ is Update} 
  \State Extract node $i \in \vs$ from INSTR$(k)$
        \State Update($i$): $d_i \leftarrow \min_{j \in \Ns(i)} d_{ij} + e_{ij}$ by \eqref{eq:update}
\ElsIf{INSTR$(k)$ is Write} 
    \State Extract edge $(i,j) \in \es$ from INSTR$(k)$
        \State Write($i,j$): $m_{ij} \leftarrow d_j$ by \eqref{eq:write}
\ElsIf{INSTR$(k)$ is Read} 
        \State Extract edge $(i, j) \in \es$ from INSTR$(k)$
        \State Read($i, j$): $d_{ij} \leftarrow m_{ij}$ by \eqref{eq:transmit}
\EndIf 
\EndFor
\State $t \leftarrow t+1$
\EndLoop
\end{algorithmic}
\end{algorithm}

The asynchronous ABF algorithm we consider is summarized in Algorithm~\ref{alg:path_alg_nd} below. The execution of Algorithm~\ref{alg:path_alg_nd} is deterministic given the sequences $(\mathcal{R}(t), \mathcal{U}(t), \mathcal{W}(t),\mathcal Q(t))$. The execution order of $\mathcal{Q}(t)$ (while loop) is intended as occurring in a fast time scale $\rw{\delta \ll 1/\abs{\mathcal{Q}(t)}}$, such that the first instruction of $\mathcal{Q}(t)$ occurs at time $t - (\abs{\mathcal{Q}(t)}-1)\delta$ and the last instruction of $\mathcal{Q}(t)$ uses information from the states $D(t-\delta)$ to compute $D(t)$. \rw{The fast time scale $\delta$ is a  formalism  solely used to resolve race conditions in execution.} The notation $D(t)$ will be used to denote the states of $(\{d_i\} , \{d_{ij}\}, \{m_{ij}\})$ after all instructions in $\mathcal{Q}(t)$ have finished.

Permutations of instruction orderings in the execution queue $\mathcal{Q}(t)$ could lead to different behaviors, as shown in Example~\ref{ex:asynchrony} below.  In our mathematical model, the ordering of the instructions in $\mathcal{Q}(t)$ is arbitrary; in practice, the order would be chosen  depending on the physical implementation of the algorithm (possibly accounting for stochastic shuffling). 
For instance, a SORT-based scheme could be implemented, in which all Writes occur after all Updates and before all Reads (or other permutation).

\begin{rmk}
\label{rmk:synchronous}
    The discrete-time dynamical law in \eqref{eq:abf} (synchronous distance computation) can be treated as an instance of Algorithm \ref{alg:path_alg_nd} under the SORT-based scheme (Update,  then Write, then Read) with
        $\forall t: \ \mathcal{R}(t), \mathcal{W}(t)= \es, \  \mathcal{U}(t) = \vs.$    
\end{rmk}

\begin{exmp}\label{ex:asynchrony}
    Consider  the graph in Figure \ref{fig:two_node} with $S = 1$. 

    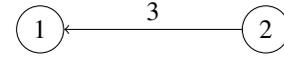
\begin{figure}[h]
        \centering

    \begin{tikzpicture}
    
    \node[shape=circle,draw=black] (A) at (0,0) {1};
    \node[shape=circle,draw=black] (B) at (3,0) {2};        
    \draw[->] (B) -- node[above] {3} (A);
\end{tikzpicture}
\caption{A two-node graph}
        \label{fig:two_node}
\end{figure}
    
    The initial conditions for this example at time $t=0$ are $d_1(0) = 40, m_{21}(0) = 30, \ d_{21}(0) = 20, \  d_2(0) = 10$. The subtables in Table \ref{tab:single_exec} outline 3 out of the 4! = 24 possibilities for $d(1)$ when allowing nondeterminism in execution under $\ \mathcal{R}(1) = \es, \Us(1) = \vs, \ \mathcal{W}(1) = \es$. \rw{Each subtable outlines the execution of Algorithm \ref{alg:path_alg_nd} between time steps $0$ and $1$. The INSTR column heading shows the action taken in the execution queue. The columns $d_2(t), d_{21}(t), m_{21}(t), d_1(t)$ display races of the respective states.}
    The column titled `Error' lists the maximum absolute value of any difference between the state in $D$ and the true value $D^*$. The ordering in the topmost subtable produces the true distance values of $d_1^* = 0, \ d_2^* = 3$ within 1 time step, \rw{the center and bottom} orderings require more iterations to reach these true values.
    \begin{table}[h]
    \centering
\caption{Possible single-iteration executions of Algorithm \ref{alg:path_alg_nd} with $\Us(1) =  \vs, \ \Ws(1) = \mathcal{R}(1) = \es$ for the graph in Figure \ref{fig:two_node}}
\label{tab:single_exec}

\begin{tabular}{lllllll}
Time $t$       & INSTR         & $d_2(t)$ & $d_{21}(t)$ & $m_{21}(t)$ & $d_1(t)$ & Error \\ \hline
$0$         & Initial       & 10    & 20  & 30      & 40 &  40 \\
$1-3\delta$ & Update(1)     & 10    & 20  & 30 & 0     &   30    \\
$1-2\delta$ & Write(2,1)     & 10    & 20  & 0    & 0   & 20 \\
$1-1\delta$  & Read(2,1) & 10    & 0 & 0     & 0    & 7 \\
$1$         & Update(2)  & 3    & 0     & 0        & 0 & 0
\end{tabular} 
\vspace{0.25cm}

\begin{tabular}{lllllll}
Time $t$       & INSTR         & $d_2(t)$ & $d_{21}(t)$ & $m_{21}(t)$ & $d_1(t)$  & Error\\ \hline
$0$         & Initial       & 10    & 20  & 30      & 40   & 40\\
$1-3\delta$ & Write(2, 1)     & 10    & 20  & 40     &40  & 40\\
$1-2\delta$ & Update(1)     & 10    & 20  & 40 & 0    & 40 \\
$1-1\delta$  & Read(2,1) & 10    & 40   & 40 & 0   & 40  \\
$1$         & Update(2)  & 43    & 40     & 40        & 0 & 40
\end{tabular} 
\vspace{0.25cm}

\begin{tabular}{llllllll}
Time $t$       & INSTR         & $d_2(t)$ &  $d_{21}(t)$ & $m_{21}(t)$ & $d_1(t)$ & Error\\ \hline
$0$         & Initial       & 10    & 20  & 30      & 40 & 40  \\
$1-3\delta$ & Update(1)     & 10    & 20  & 30 & 0      & 30 \\
$1-2\delta$ & Read(2,1)     & 10    & 30  & 30     & 0    & 30 \\
$1-1\delta$  & Write(2,1) & 10    & 30   & 0 & 0   & 30    \\
$1$         & Update(2)     & 33    & 30 & 0 & 0   & 30 
\end{tabular} 
\end{table}

\end{exmp}

\begin{rmk}[Practical implementation]\label{rmk:asyncmodel} 
\rw{
The discrete event system representation can be used to model phenomena such as communication and computation delays, as exemplified in Figure~\ref{fig:asyncmodel}
for agents $i$ and $j$ in asynchronous execution.}

\begin{figure}
\includegraphics[width=\linewidth]{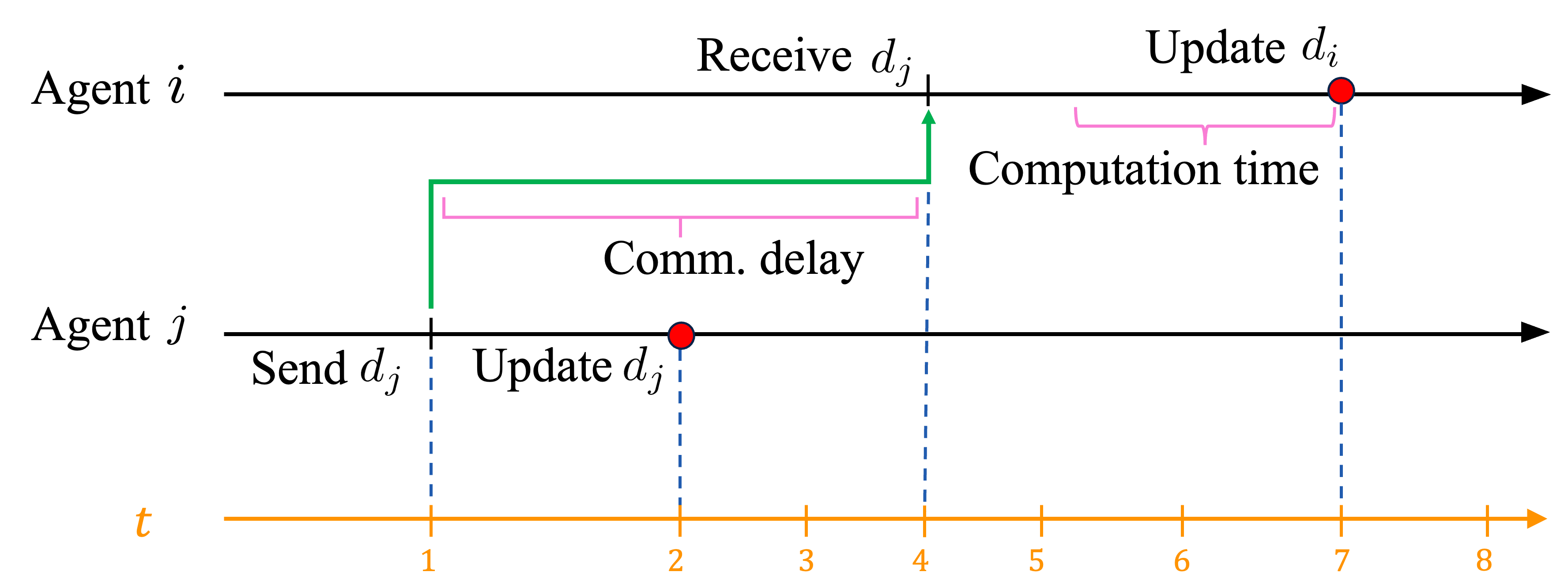}
        \caption{Example of asynchronous communication and updates.}\label{fig:asyncmodel} 
\end{figure}

\rw{ In practice, agent $j$ sends $d_j$ to agent $i$ over a wireless communication network, and agent $j$ then updates its distance $d_j$. Agent $i$ receives $d_j$ with some delay, and performs an update of $d_i$ with non-negligible computation time, using the latest (outdated)  received value of $d_j$. Agents $i$ and $j$ do not have a common clock in their execution. Mathematically, we can  model this situation by using the fictitious outbox buffer, the actions (Read, Update, Write),  and an event counter $t$ (increased every time any of the $N$ agents performs an action). The situation in Figure \ref{fig:asyncmodel} can be described by the following actions: at time $t=1$    Write($i,j)$; at time $t=2$ Update($j$); at time $t=4$ Read($i,j$); at time $t=7$ Update$(i)$. 
    \\
        The description is not unique (e.g., Read($i,j$)  and Update$(i)$ could both happen at time $t=7$, preserving this order in the Queue $\mathcal Q(7)$). In this setup, Assumption A2 is satisfied as long as inter-update /communication  intervals and communication/computation delays are bounded  without out-of-order messaging  \cite[Ch.~6, Ex.~1.1]{Bertsekas1989}.
Note that 
}  the internal buffer $d_i$ and the  inbox buffer $d_{ij}$ are the information effectively stored in the local memory of agent $i$. The  outbox buffer $m_{ij}$ is a fictitious extra variable only used in the mathematical model.
 \end{rmk}

\subsection{Assumptions}
The following assumptions will be imposed:
\begin{itemize}
    \item[\textbf{A1}] There exists a directed path in $\gs$ from each node $i \in \vs$ to $S$ and $S \subset \vs$.
    \item[\textbf{A2}] 
    There exists known finite time window lengths $P_\Rs \geq 0$, $P_\Us \geq 1$, $P_\Ws \geq 0$ such that,  for all $t \in \N$: 
    $ \cup_{k=0}^{\max(0, P_\Rs-1)} \mathcal{R}(t+k)  = \es$; 
    $\cup_{k=0}^{P_\Us-1} \mathcal{U}(t+k)   = \vs$; and $ \cup_{k=0}^{\max(0, P_\Ws-1)} \mathcal{W}(t+k)  = \es$. 
 
\end{itemize}

Assumption A1 ensures that all true distances $d^*_i$ are finite. Assumption A2 is required to ensure finite-time  convergence times in the asynchronous setting. 
\rw{Under Assumption A2, each node Updates every $P_\Us$ time steps, and each edge Reads and Writes every $P_\Rs$ and $P_\Ws$ steps, respectively.}
We  adopt a convention where $P_\Ws=0$ implies Writes occur after all Update instructions and before all Read instructions in the queue (in arbitrary order); and $P_\Rs =0$ implies Reads occur after all non-Read instructions
(in a likewise arbitrary order). In particular, note that $P_\Ws = P_\Rs =0$, $P_\Us = 1$ corresponds to the synchronous implementation  in Remark~\ref{rmk:synchronous}, where the actions happen in the order Update, then Write, then Read.

Furthermore, we  define the asynchrony measure
\begin{align}
    P \coloneqq P_\Rs + P_\Us + P_\Ws. \label{eq:p_time_window}
\end{align}
\rw{E}ach node will Update and Write to all neighbors,  and its neighbors will Read this  updated estimate every $P$ time instances. \rw{$P$ is therefore an upper-bound on the rate of new local information propagation.} This $P$-reset is noted in the proof of \cite[Proposition 1]{bertsekas1982distributed}.

\begin{rmk}
    The work   \cite{bertsekas1982distributed} first proved  finite-time convergence 
    for an asynchronous implementation of \eqref{eq:abf}, noting that
      the convergence time $T$ is `non-polynomial' in the graph parameters. However, exact bounds are not derived. Furthermore, in \cite{bertsekas1982distributed}  it is  assumed that executions are blocking: no more than one event per node  (Update or Transmit) could occur at each time instant (though simultaneous Transmission to multiple neighbors is possible). The computation model in Algorithm \ref{alg:path_alg_nd} allows for multiple possible actions at each time instant through the execution queue $\mathcal{Q}(t).$
\end{rmk}

\section{Finite-time convergence Bounds}
\label{sec:convergence}

This section presents upper-bounds on convergence times for shortest-distance computation.

\begin{defn}
\label{def:converge}
    The distance estimates $D(\cdot)$ are said to  {converge in finite time} with convergence bound $T\in \N$ if   $D(T) = D^*$, \rw{with} $D^*$ as in  \eqref{eq:steady_state}.
\end{defn}

\subsection{Overestimates and Underestimates}

The works  \cite{dasgupta2016lyapunov, mo2019robustness} analyzed convergence of the Bellman--Ford method in terms of bounding overestimates and underestimates of $d^*.$ To deal with asynchrony, we should also consider errors on the buffer variables, as done next.  

\subsubsection{Estimate Definitions}

The distance error of the distance estimates  $d_i(t)$, $m_{ij}(t)$, and $d_{ij}(t)$ along with the worst-case over and underestimates are: \begin{align}
    \Delta_i(t) &= d_i(t) - d^*_i, & \Delta^{\pm}_\Us(t) &= \max_{i \in \vs}(\pm \Delta_i(t), 0), \nonumber \\ \Delta_{ij}^m(t) &= m_{ij}(t) - d^*_{\rw{j}},\!\! & \Delta^{\pm}_\Ws(t) &= \max_{(i, j) \in \es}(\pm \Delta_{ij}^m(t), 0), \label{eq:delta_bound}\\ \nonumber \Delta_{ij}(t) & = d_{ij}(t) - d^*_j, & \Delta^{\pm}_{\Rs}(t) &= \max_{(i, j) \in \es}(\pm \Delta_{ij}(t), 0).
    \end{align}

We note that $\Delta^+_\Us$, $\Delta^+_\Ws$, $\Delta^+_\Rs$, $\Delta^-_\Us$, $\Delta^-_\Ws$, and $\Delta^-_\Rs$ are nonnegative functions of $t$.

The worst-case estimates in the entire network are:
\begin{subequations}
\label{eq:worst_case_estimates}
\begin{align}
    \Delta^+(t) & = \max \left[ \Delta_\Us^+(t), \Delta_\Ws^+(t), \Delta_\Rs^+(t)\right]\\
     \Delta^-(t) &= \max \left[ \Delta_\Us^-(t), \Delta_\Ws^-(t),\Delta_\Rs^-(t)\right].
\end{align}
\end{subequations}

In the asynchronous setting, we refer to a Lyapunov function as a function that is positive definite about a fixed point and monotonically decreases along system trajectories. 
The errors $\Delta^+$ and $\Delta^-$ just defined can be used to construct  the following (candidate) Lyapunov functions:
\begin{subequations}
    \label{eq:lyap}
\begin{align}
L_+(t) &= \Delta^+(t) + \Delta^-(t),\label{eq:lyap_sum} \\
    L(t) &= \max \left[\Delta^+(t), \Delta^-(t) \right]. \label{eq:lyap_inf}    
\end{align}
\end{subequations}

In particular, our proposed Lyapunov function $L(t)$ is the  $L_\infty$ norm over distance errors, since $L(t)= \norm{D(t) - D^*}_\infty$.

\subsection{Some definitions}

To express our convergence bounds, we require the notion of effective diameter; see Appendix~\ref{app:defn} for a formal formulation. 

\begin{defn}[\cite{mo2019robustness}]
\label{def:effectiveDiameterInformal}
    The effective diameter $\mathcal{D}(\gs)$ is the maximum number of nodes in any shortest path from \rw{any vertex $i \in \vs$ to any vertex in $S$}.
\end{defn}

\begin{exmp}
\label{exmp:buckyball_start}
    A recurring example throughout this paper will be the Buckyball graph in Figure \ref{fig:Buckyball}. This graph has  60 vertices and 180 edges and has the adjacency pattern of a soccer ball/football. The source set of $S = \{19, 37, 57\}$ are marked as green dots. Integer weights of the directed edges are picked uniformly at random in the range of $1..20$ (with $e_{\min} = 1$). It is possible  that $e_{ij} \neq e_{ji}$ between adjacent nodes $(i, j)$. The thickness of the edges in Figure \ref{fig:buckyball_vis} are proportional to the value of the weights (a thicker edge indicates a larger weight $e_{ij}$). The highlighted red edges in Figure \ref{fig:buckyball_vis} are edges $(i, j) \in \es$ that are found on some shortest path starting from $k \in \vs$ and ending in $S$, while the black edges  are not traversed in any shortest path. Overlapping black and red edges indicate that a shortest path exists in one direction (red) but not in reverse (black).

    Figure \ref{fig:buckyball_tcn} draws all the possible shortest paths from each node to the source set,    
    for the Buckyball in Figure \ref{fig:buckyball_vis} ---the so-called true constraining graph, see Appendix~\ref{app:defn}. 
     The \rw{s}ource nodes are at the bottom of Figure \ref{fig:buckyball_tcn}.
    This true constraining graph is acyclic but is not a tree: the origin node 54 can travel to $19 \in S$ with a shortest-path cost of $d^*_{54} = 25$ via the routes $(54, 53, 19)$ or $(54, 31, 32, 18, 19)$.
     The other permissible branching nodes in Figure \ref{fig:buckyball_tcn} are nodes 15 and 49. The maximum number of nodes in any shortest-path is 9, and is achieved by the 9-length sequences $(27, 25, 24, 48, 49, 50, 59, 58, 57)$ and $(27, 25, 24, 48, 49, 51, 52, 53, 19)$. Hence, the effective diameter is  $\mathcal D({\mathcal G}) = 9$ in this example. 

    \begin{figure}[!h]
        \centering
         \centering
    \begin{subfigure}[b]{0.25\textwidth}
    \centering
    \includegraphics[width=\linewidth]{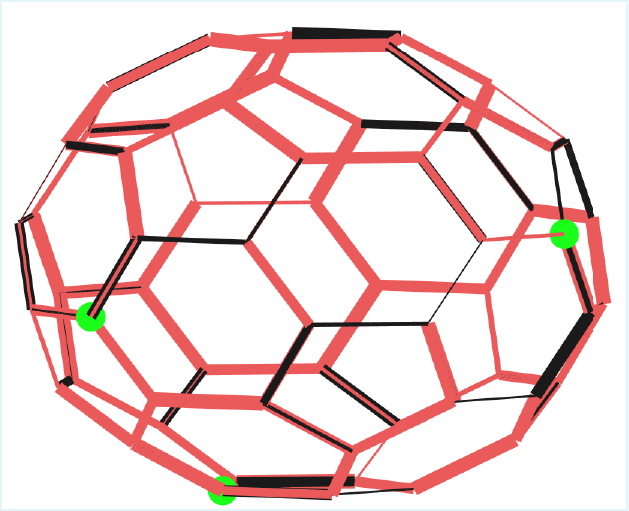}
        \caption{The Buckyball graph        }
    \label{fig:buckyball_vis}
    \end{subfigure}
    \vspace{0.3cm}
    
        \begin{subfigure}[b]{0.45\textwidth}
        \centering
\includegraphics[width=0.75\linewidth,height=6cm]{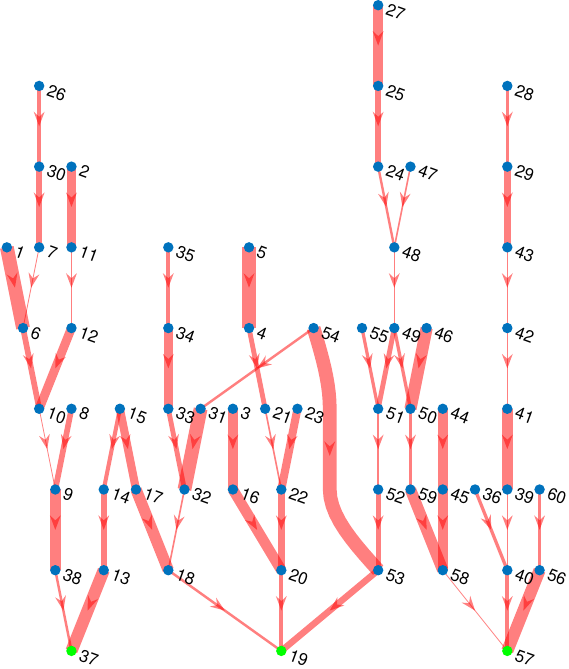}
        \caption{Edges involved in shortest-paths to $S$}
    \label{fig:buckyball_tcn}
    \end{subfigure}
    \caption{Buckyball graph structure}
    \label{fig:Buckyball}
    \end{figure}
\end{exmp} 

\vspace{0.5em}
To pose the underestimate convergence bound, we define the function $D^{\min}(t)$ to denote the minimal value underestimate of any state or edge in $D(t)$ to $D^*$:
\begin{equation}
    D^{\min}(t) = \begin{cases}
        \displaystyle \min_{k \text{ s.t. } [D(t)]_k \leq D^*_k}\ [D(t)]_k&  \text{if } \exists k:  [D(t)]_k < D^*_k \\
        \infty & \text{otherwise}.
    \end{cases}\label{eq:dmin_D}
\end{equation}

Finally, let
\begin{equation}
d^*_{\max} = \max_{i \in \vs} d^*_i
\end{equation}
 be the largest distance from any agent to $S$ (bounded by A1). 

\subsection{Finite-Time Convergence Bounds}

We are ready to prove finite-time convergence bounds for the asynchronous ABF algorithm in Algorithm~\ref{alg:path_alg_nd}. \rw{Compared to \cite{mo2019robustness,mo2021mpp}, we address the technical complications due to the presence of the extra Read and Write operations
.} 

We start  by proving that the worst-case estimates in \eqref{eq:worst_case_estimates} are monotonically non-increasing in $t$. \rw{Interestingly,  this holds because our asynchrony model excludes out-of-order messages (e.g., agent $i$ cannot write  $d_i(0)$ after $d_i(3)$), differently from the partially asynchronous scenario \cite{Bertsekas1989,bertsekas1982distributed}}.

\begin{lem}
\label{lem:monotonic_alt}
    The worst-case estimate sequences  $\{\Delta^+(t)\}_{t \in \N}$ and $\{\Delta^-(t)\}_{t \in \N}$ are monotonically non-increasing under Assumptions A1-A2.
\end{lem}
\begin{proof}
    See Appendix \ref{app:monotonic_alt}.
\end{proof}

\begin{thm}
\label{thm:ultimate_deltaplus}
Under Assumptions A1-A2, the overestimate $\Delta^+$ has the convergence bound:
\begin{align} \forall t &\geq    T^+_P \coloneqq P \Dc(\gs): \qquad  &\Delta^+(t) = 0.\label{eq:deltaplus_ultimate_plus}
    \end{align}
\end{thm}
\begin{proof}
   See Appendix \ref{app:ultimate_deltaplus}.
\end{proof}

\rw{The underestimate convergence bound} may be given in terms of the minimal edge-weight \rw{$e_{\min} := \min_{(i, j) \in \es} e_{ij}.$} 
\begin{thm}
\label{thm:ultimate_deltaminus}
    Under Assumptions A1-A2, the underestimate $\Delta^-$ has the convergence bound:
\label{eq:deltaminus_ultimate}
    \begin{align}
        \forall t & \geq T_P^{-} \coloneqq P \left \lceil \frac{d^*_{\max} - D^{\min}(0)}{e_{\min}} \right \rceil: \qquad  &\Delta^-(t) = 0, \label{eq:deltaminus_ultimate_delay}
    \end{align}
    \vspace{0.1em}
    \rw{with the convention $\lceil -\infty \rceil = -\infty$.}
    \vspace{0.1em}
\end{thm}
\begin{proof}
See Appendix \ref{app:ultimate_deltaminus}.
\end{proof}

In summary,  the bounds from Theorems \ref{thm:ultimate_deltaplus} and \ref{thm:ultimate_deltaminus} are
\begin{align}
\label{eq:ultimate_summary}
    T^+_P &=  P \Dc(\gs) & 
    T^-_P &= P \left \lceil \frac{d^*_{\max} - D^{\min}(0)}{e_{\min}} \right \rceil. 
\end{align}

The above convergence bounds 
in \eqref{eq:ultimate_summary} 
are simply $P$ times the convergence bounds in \cite{mo2019robustness}. We note that 
the convergence bound  for $\Delta^+$ is a geometric quantity that depends only on the graph $\gs$ and its weights. The convergence bound for $\Delta^-$ further depends on the initial conditions $d(0)$. The phenomenon where the underestimate errors will generally fall much slower than the overestimate errors is referred to in \cite{Beal2008} as the ``rising value problem,'' which is related to a ``count to infinity'' in network routing \cite{cassandras2008introduction}.

\begin{cor}
\label{cor:ultimate}
A convergence bound $T_P$ for  asynchronous directed distance estimation is:
\begin{align}
\label{eq:conv_bound}
    T_P = \begin{cases}
        \max(T^+_P, T^-_P) & \s^-(0) \neq \varnothing \\
        T^+_P & \s^-(0)= \varnothing, \\
    \end{cases}
\end{align}
such that $\forall t \geq T_P: \ \Delta^+(t), \Delta^-(t), L(t), L_+(t) = 0$.
\end{cor}
\begin{proof}
    When $\s^-(0) \neq \varnothing$, this convergence bound follows from taking the maximum of the overestimate bound of Theorem \ref{thm:ultimate_deltaplus} and the underestimate bound of Theorem \ref{thm:ultimate_deltaminus}. When $\s^-(0) = \varnothing$, then monotonicity of $\Delta^-$ (Lemma \ref{lem:monotonic_alt}) indicates underestimates will never develop, and thus convergence will occur according to the $\Delta^+$  bound. From the definition of the Lyapunov functions in \eqref{eq:lyap}, $\Delta^+(t), \Delta^-(t)= 0$ implies $L(t), L_+(t) = 0$.
\end{proof}

\subsection{Buckyball Example for Convergence Bounds}
\label{sec:buckyball_ultimate}

We present examples of these asynchronous convergence bounds for the Buckyball scenario (Example \ref{exmp:buckyball_start}, Figure \ref{fig:Buckyball}), each involving the initial condition
\begin{equation}
    \forall i \in \vs: \ d_i(0) = 0, \quad \forall (i, j) \in \es: \ m_{ij}(0) = d_{ij}(0) = \infty. \label{eq:buckyball_init}
\end{equation}

At $P_\Rs=4, P_\Us = 4, P_\Ws = 2$ with $P=10$, the  convergence bounds for the initial condition \eqref{eq:buckyball_init} are $T^+ = 80, \ T^- = 460$. The maximal distance to $S$ from any node of the Buckyball graph is $d^*_{\max} = 41$.

Figure \ref{fig:bucky_multi} visualizes $\Delta^+(t)$ and $\Delta^-(t)$ over 20 executions of Algorithm \ref{alg:path_alg_nd} with $P=10$ for the Buckyball system. Each trajectory starts at the same initial condition \eqref{eq:buckyball_init}. The trajectories may differ in their external signals $(\mathcal{R}(t), \mathcal{U}(t), \mathcal{W}(t))$. Furthermore, the order in the queue $\mathcal Q(t)$ is permuted randomly. The gray region in the top subplot of Figure \ref{fig:bucky_multi} denotes a ``forbidden'' region based on the convergence bound: the convergence bound of $T^+=90$ rules out $\Delta^+(t) > 0$ past $t \geq 90$. The overestimate $\Delta^+(t)$ falls below $\infty$ at time $13$ for each of the plotted trajectories, and thus $\Delta^+(t)$ is not displayed between times 0 and 9.  This gray forbidden region is omitted on the bottom subplot, because in the 50 executions, all sampled underestimates vanish to zero within time 91 (before $T^- = 460$).

\begin{figure}[!h]
    \centering
    \includegraphics[width=\linewidth]{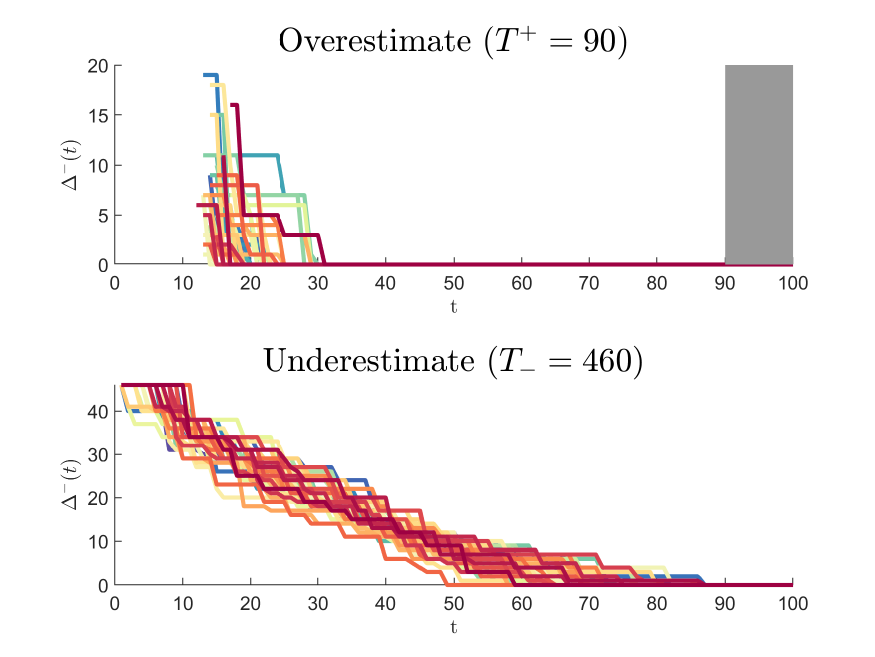}
    \caption{50 Trajectories of Buckyball execution each with random $\mathcal{Q}(\cdot)$ ordering schemes at $P=10$}
    \label{fig:bucky_multi}
\end{figure}

\section{Robustness}
\label{sec:robustness}

Section \ref{sec:convergence} analyzed convergence behaviors of the asynchronous ABF scheme without noise. In this section, we further  consider the presence of three kinds of noise: communication noise, memory noise, structural noise. Communication and memory noise  corrupt the Read and Write steps and model uncertainty on the message passing, while structural noise affects the Update steps and  models uncertainty on the geometry of the underlying graph and weights (such as through agents moving or vibrating). 

\subsection{Noise Model}
At each time $t \in \N$, let $\epsilon^\Rs_{ij}(t)$ be the communication (Read) noise on edge $(i,j)$, let 
$\epsilon^\Us_{ij}(t)$ be the structural (Update) noise on edge $(i,j)$, and let $\epsilon^{\Ws}_{ij}(t)$ be the memory (Writing) noise on edge $(i,j)$.
The noisy version of the Read, Update, and Write dynamics in \eqref{eq:update} - \eqref{eq:transmit} respectively are
\begin{subequations}
\label{eq:noisydynamics}
\begin{align}
d_{ij} &\leftarrow m_{ij} + \epsilon^\Rs_{ij},             \label{eq:Read_eps} \\
              d_{i} &\leftarrow \begin{cases}
    \min_{j \in \Ns(i)} d_{ij} + e_{ij} +\epsilon_{ij}^\Us & i \not\in S \\
    0 & i \in S,
\end{cases}  \label{eq:update_eps} \\
m_{ij} &\leftarrow d_{j} + \epsilon^\Ws_{ij}.              \label{eq:write_eps}
    \end{align}
\end{subequations}
The noisy asynchronous ABF algorithm is obtained by replacing  the steps  \eqref{eq:update}-\eqref{eq:transmit}   in Algorithm~\ref{alg:path_alg_nd} with the steps  \eqref{eq:Read_eps}-\eqref{eq:write_eps} respectively.

Here we  consider  (possibly non-symmetric) bounds for the  noise given by
\begin{align}
   & \epsilon^\Rs_{ij}(t) \in [\epsilon^\Rs_{\min}, \epsilon^\Rs_{\max}], \quad  \epsilon^\Us_{ij}(t) \in [\epsilon^\Us_{\min},\nonumber  \epsilon^\Us_{\max}] \\
   &\epsilon^\Ws_{ij}(t) \in [\epsilon^\Ws_{\min}, \epsilon^\Ws_{\max}]. \label{eq:noise_general}
\end{align}

The synchronous work in \cite{mo2019robustness} did not include Read nor Write steps, and thus did not possess communication noise $\epsilon^\Rs$ or memory noise $\epsilon^\Ws$.
The authors in \cite{mo2019robustness} only addressed two-sided noise as $\epsilon^\Us(t) \in [-\bar{\epsilon}^\Us, \bar{\epsilon}^\Us]$, and imposed the condition that $\bar \epsilon^{\Us}< e_{\min}$ in order to prevent negative-valued distances from appearing. 
We note that one-sided nonnegative noise does not  require the restriction that $\epsilon^\Us < e_{\min}$. 
Since we consider Update, Read  and Write noise, we make the following assumption:

\begin{itemize}
    \item[\textbf{A4}] For all $(i,j) \in \mathcal E$, for all $t\geq 0$ the bounding relationship in \eqref{eq:noise_general} holds, where  $0 \in [\epsilon^\Us_{\min}, \epsilon_{\max}^\Us]$, $0 \in [\epsilon^\Rs_{\min}, \epsilon^\Rs_{\max}]$,  $0 \in [\epsilon^\Ws_{\min}, \epsilon^\Ws_{\max}]$, and $e_{\min} + \epsilon^\Us_{\min} + \epsilon^\Rs_{\min} + \epsilon^\Ws_{\min}> 0$.\vspace{0.3em}
\end{itemize}

In the following, we will also define
\begin{align}
    \epsilon_{\max} &= \epsilon_{\max}^\Rs + \epsilon_{\max}^\Us+ \epsilon_{\max}^\Ws \nonumber\\
    \epsilon_{\min} &= -(\epsilon_{\min}^\Rs + \epsilon_{\min}^\Us+ \epsilon_{\min}^\Ws).
\end{align}

\subsection{Robustness Bounds}

The behavior of shortest path computation under the disturbances $\epsilon^\Rs, \epsilon^\Us, \epsilon^\Ws$ can be analyzed with respect to the following extreme laws for the Read action
\begin{subequations}
    \label{eq:minandmaxupdates}
\begin{align}
d_{ij}^+ &\leftarrow m_{ij}^+ + \epsilon^\Rs_{\max}, \qquad
d_{ij}^-   \leftarrow m_{ij}^- + \epsilon^\Rs_{\min},              \label{eq:Read_eps_2}
\intertext{the Update action}
              d_{i}^+ &\leftarrow \begin{cases}
    \min_{j \in \Ns(i)} d_{ij}^+ + e_{ij} +\epsilon_{\max}^\Us & i \not\in S \\
    0 & i \in S,  
\end{cases}  
\nonumber \\
d_{i}^- &\leftarrow \begin{cases}
    \min_{j \in \Ns(i)} d_{ij}^- + e_{ij} +\epsilon_{\min}^\Us & i \not\in S \\
    0 & i \in S,
    \end{cases} \label{eq:update_eps_bound} \\
    \intertext{and the Write action}    
m_{ij}^+ &\leftarrow d_{j}^+ + \epsilon^\Ws_{\max}, \qquad
m_{ij}^-   \leftarrow d_{j}^- + \epsilon^\Ws_{\min}.              \label{eq:Write_eps_2}            
    \end{align}
    \end{subequations}

We will show that the behavior of $D(t)$ under uncertainties $(\epsilon^\Rs, \epsilon^\Us, \epsilon^\Ws)$  can be ``sandwiched'' between the bounds $d^\pm(t)$:

In particular, given the graph $\gs$ and noise bounds \rw{$(\epsilon_{\min}^{\Rs}, \epsilon_{\min}^{\Us}, \epsilon_{\min}^{\Ws})$ satisfying the inequality} $\epsilon_{\min}^{\Rs}+ \epsilon_{\min}^\Us + \epsilon_{\min}^\Ws+ e_{\min}>0$, let $\gs^-$ and $\gs^+$ be the graph obtained by replacing each edge distance $e_{ij}$ with the positive distances $e_{ij} - \epsilon_{\min}$ and $e_{ij} + \epsilon_{\max}$, respectively. 

The robustness and convergence bounds for the noisy asynchronous ABF  are described by Theorem \ref{thm:robust}.

\begin{thm}
\label{thm:robust}
    Under Assumptions~A1-A4, the overestimates  and underestimates generated by the noisy asynchronous ABF algorithm are bounded as 
        \begin{align}
        \label{eq:robust_plus_bound}
        \Delta^+(t) & \leq B^+ & & \forall t \geq P\mathcal{D}(\gs^+) \\
        \label{eq:robust_bound}
 \Delta^-(t) & \leq  B^- & &             \forall t \geq P \left\lceil\frac{d_{\max}^*(\gs^-) - D^{\min}(0)}{e_{\min} - \epsilon_{\min}}\right\rceil, 
 \end{align}
 where
 \begin{align}
    B^+ &= (\mathcal{D}(\gs)-1)\epsilon_{\max}, & B^- &= (\mathcal{D}(\gs^-)-1)\epsilon_{\min}. \label{eq:error_bounds}
\end{align}
    
    \vspace{0.1em}
\end{thm}
\begin{proof}
See Appendix \ref{app:lem_robustness}.
\end{proof}

\begin{rmk}
    The error bounds $B^\pm$ in \eqref{eq:error_bounds} are independent of the asynchronous time-window length bound $P$. The value of $P$ only affects the convergence time for steady-state error.
\end{rmk}

The Lyapunov functions can also be used to provide an intuition of the `gain' between the applied noise and the distance estimates in the sense of Input-to-State Stability \cite{Sontag2008}.
 The following gain-type behavior is therefore observed:
\begin{cor} Let Assumptions~A1-A4 hold, and consider the noisy asynchronous ABF algorithm with actions \eqref{eq:noisydynamics}. Let $T^-_P$ be the finite-time convergence bound from \eqref{eq:robust_bound},  $T^+_P$ be the finite-time convergence bound from \eqref{eq:robust_plus_bound}, and $T_P \coloneqq {\max}[T^+_P,T^-_P]$. 
In the one-sided nonnegative case $\epsilon_{\min}=0$,  it holds that:
\begin{align*}
   \forall t \geq T_P: & &L_+(t) & \leq (\Dc(\gs)-1)\epsilon^{\max}  \\
  \forall t \geq T_P: & &  L(t) & \leq (\Dc(\gs)-1)\epsilon^{\max}.
    \intertext{In the one-sided nonpositive case $\epsilon_{\max} = 0$, it holds that}
    \forall t \geq  T_P: & & L_+(t) & \leq (\Dc(\gs^-)-1)\epsilon^{\min}\\
     \forall t \geq T_P: & & L(t) & \leq (\Dc(\gs^-)-1)\epsilon^{\min} .
    \end{align*}
 In the two-sided case noise, with $\epsilon_{\max} \in (0, e_{\min}), \ \epsilon_{\min} = \epsilon_{\max}$, it holds that, $\forall t \geq  T_P$, 
 \begin{align*}
  L_+(t) &  \leq (\Dc(\gs) + \Dc(\gs^-) -2)\epsilon^{\max}\\
  L(t) & \leq (\max[\Dc(\gs), \Dc(\gs^-)]-1)\epsilon^{\max}.
\end{align*}
    
\end{cor}
\begin{proof}
This follows from Theorem \ref{thm:robust} by recalling the Lyapunov definitions in \eqref{eq:lyap}.    
\end{proof}

\subsection{Buckyball Examples for Noise Bounds}
This subsection performs a robustness analysis of the Buckyball example, continuing Example \ref{exmp:buckyball_start} and Section \ref{sec:buckyball_ultimate}. The initial condition used for this analysis remains \eqref{eq:buckyball_init}. 
Figure \ref{fig:bucky_noisy} plots the evolution of distance overestimates and underestimates for a single trajectory with $P_\Rs = 2, P_\Us = 2, P_\Ws=1$ for the noise process of  $[\epsilon_{\min}^\Rs, \epsilon_{\max}^\Rs] = [0, 2]$,
$[\epsilon_{\min}^\Us, \epsilon_{\max}^\Us] = [-0.3, 0]$, and $[\epsilon_{\min}^\Ws, \epsilon_{\max}^\Ws] = [-0.1, 0.1]$. This asynchrony and noise setting can be summarized by the quantities $P=5,$ $\epsilon_{\min} = 0.4$, and $\epsilon_{\max} = 2.1$. The ultimate bounds on robustness are $T^+ = 45, B^+ = 18.9, \ T^- = 385, B^- = 3.6$.

 \begin{figure}
     \centering
     \includegraphics[width=0.9\linewidth]{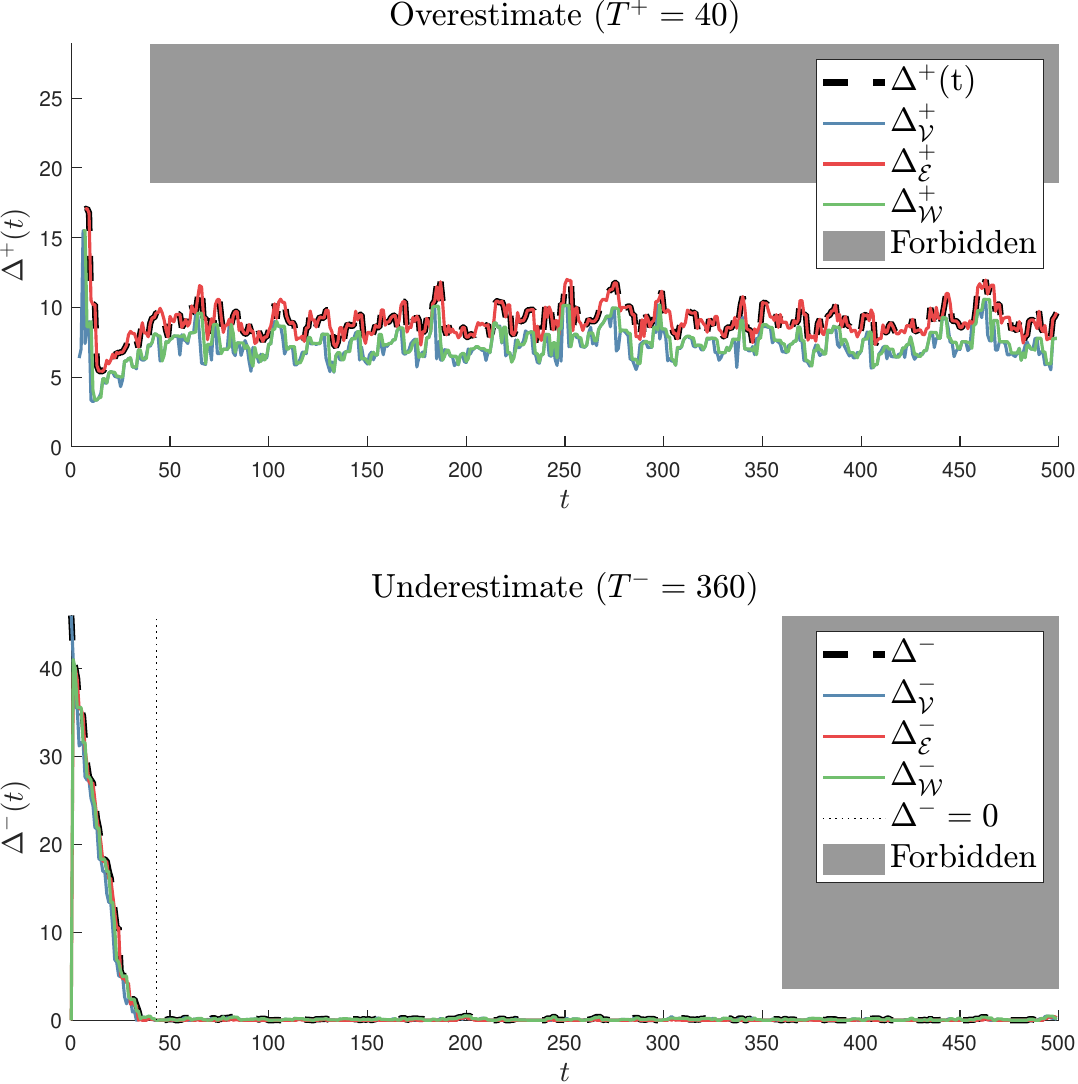}
     \caption{Buckyball execution under noise \label{fig:bucky_noisy}}     
 \end{figure}

\section{Most Probable Path}
\label{sec:mpp}

The most probable path (MPP) problem \cite{khalife2008probabilistic} can be solved through similar methods as  the shortest-path method. In the MPP, the edge weight between nodes $(i, j)$ of each link  corresponds to the probability that a message will successfully be sent from node $i$ to node $j$. This weight is expressed as the probability $p_{ij} \in (0, 1)$ for all $(i, j) \in \es$. If nodes $(i, j)$ have a successful-message-passing probability of $1$ between them, then nodes $(i, j)$ can be merged together.

\begin{prob}
\label{prob:mpp}
    Given a graph $\gs(\vs, \es)$ with source set $S$, edge probabilities $e \in (0, 1)$, and a starting node $i$, solve:
    \begin{align}
        \text{Prob}(i) = \max_{\mathcal{P} \in \text{Path}(i, S)} \prod_{(v, v') \in \mathcal{P}} p_{v, v'}. \label{eq:mpp_prob}
    \end{align}

\end{prob}

The most probable path $\mathcal{P}$  solving \eqref{eq:mpp_prob} offers the highest likelihood of the message surviving the voyage from $i$ to $S$.
A dynamic programming algorithm is used in \cite{khalife2008probabilistic} to solve Problem \ref{prob:mpp}. 
This synchronous law can also be expressed in terms of a probability of success $\theta_i(t)$ as (modified from the formulation in \cite{mo2021mpp} posed as a probability of failure)
\begin{align}
\theta_i(t+1) &= \begin{cases}
        \max_{j \in \Ns(i)} \theta_j(t) p_{ij} & i \not\in S \\
        1 & i \in S. \label{eq:mpp_update}
    \end{cases}
\end{align}
The work in \cite{mo2021mpp} developed finite-convergence and robustness bounds for MPP in the synchronous case.
For an asynchronous implementation of this algorithm, the Update, Read, Write laws for MPP, with node-probabilities $\theta_i$,  edge-probabilities $\theta_{ij}$, and buffers $\phi_{i,j}$ are
\begin{align}
\text{Read:}& & \theta_{ij} &\leftarrow \phi_{ij} \\
       \text{Update:}& &      \theta_i &\leftarrow \begin{cases}
        \max_{j \in \Ns(i)} \theta_{ij} p_{ij} & i \not\in S \\
        1 & i \in S \label{eq:mpp_update_async}
    \end{cases} \\
    \text{Write:}& & \phi_{ij} &\leftarrow \theta_j.
\end{align}

We note that MPP can be phrased as a shortest-path problem using a logarithmic transformation.
\begin{prop}
    Problem \eqref{eq:mpp_prob} is a shortest-path problem over the graph $\gs(\vs, \es)$ with logarithmically-transformed weights $e_{ij} = -\log p_{ij}.$
\end{prop}
\begin{proof}
The path objective from \eqref{eq:mpp_prob} can be expressed as
\begin{align}
    \text{Prob}(i) = &\max_{\mathcal{P} \in \text{Path}(i, S)} \prod_{(v, v') \in \mathcal{P}} p_{v, v'}. \\
    \intertext{Because the log function is monotonically increasing, the maximization objective is preserved under a logarithmic transformation}
    \log \text{Prob}(i)= &\max_{\mathcal{P} \in \text{Path}(i, S)} \log \left( \prod_{(v, v') \in \mathcal{P}} p_{v, v'} \right)\\
= &\max_{\mathcal{P} \in \text{Path}(i, S)} \sum_{(v, v') \in \mathcal{P}} \log{p_{v, v'}}. \\
    \intertext{We can now take the negative of both sides to obtain an equivalent minimization problem}
        -\log \text{Prob}(i) & = \min_{\mathcal{P} \in  \text{Path}(i, S)} \sum_{(v, v') \in \mathcal{P}} (-\log{p_{v, v'}}) \label{eq:equiv_min_2}, \\
        & =\min_{\mathcal{P} \in  \text{Path}(i, S)} \sum_{(v, v') \in \mathcal{P}} e_{(v, v')}, \label{eq:equiv_min}
        \end{align}
which has the form of a shortest path Problem \ref{prob:shortest} with weights $e = -\log p$.
\end{proof}

Convergence bounds from Section \ref{sec:convergence} for asynchronous execution can therefore be applied to the MPP problem.

\begin{exmp}
The finite-time convergence bounds for the MPP setting (based on \eqref{eq:ultimate_summary}) are 
    \begin{subequations}
    \label{eq:ultimate_summary_mpp}
\begin{align}
    T^+_P &=  P \Dc(\gs) \\
    T^-_P &= P \left \lceil \frac{\max_i[\log(\theta_i(0))] - \min_i[\log \text{Prob}(i)]}{-\max_{(i, j)} \log p_{ij}} \right \rceil. 
\end{align}
\end{subequations}
As expected, the results in \eqref{eq:ultimate_summary_mpp} are $P$ times the synchronous convergence bounds computed in \cite{mo2019robustness}.
\end{exmp}
\begin{exmp}
    Consider the case of multiplicative degradation along a link  $p^\epsilon_{ij} = p_{ij} \epsilon^p_{ij}$ for $\epsilon^p_{ij} \in [\epsilon^p_{\min}, \epsilon^p_{\max}]$. This multiplicative degradation will logarithmically transform into 
    \begin{align}
        e^{\epsilon}_{ij} = -\log p_{ij} \epsilon^p_{ij} = e_{ij} - \log \epsilon^p_{ij}. \label{eq:log_transform_eps}
    \end{align}

The expression in \eqref{eq:log_transform_eps} can be interpreted as an additive noise $\epsilon_{ij} = -\log \epsilon^p_{ij}$ taking on values $\epsilon_{ij} \in [-\log \epsilon^p_{\max}, -\log \epsilon^p_{\min}].$ Theorem \ref{thm:robust} can now be used to get convergence bounds on the worst-case probability of error.   
    \end{exmp}

\section{Examples}
\label{sec:examples}

Code to generate all examples is publicly available at \url{doi.org/10.3929/ethz-b-00723736}. All code was written in MATLAB R2023b. These test scripts are written in a serial manner in order to reliably and repeatedly generate random sequences used for the signals $(\mathcal{R}(t), \mathcal{U}(t), \mathcal{W}(t))$ and the execution queue  $\mathcal{Q}(t)$. The sequences $\mathcal{R}(t)$, $\mathcal{U}(t)$ and $\mathcal{W}(t)$ are generated by a counting process: after an Update  occurs, the time until the next Update is drawn uniformly from the integers $1..P_{\Us}$ (with a similar process holding for Read and Write). The instruction ordering from the execution queue $\mathcal{Q}(t)$ at each time $t$ is set according to a random permutation (via the \texttt{randperm} command in MATLAB).

We consider a random geometric graph  with 1000 agents uniformly distributed inside the box $[0, 6000] \times [0, 8000] \times [0, 10000]$, plotted in Figure \ref{fig:space_position}.
The source set $S$ comprises the first 10 agents. Each agent has 5 outgoing edges corresponding to its 5 closest neighbors in the Euclidean metric. The weight of each edge is the Euclidean distance between the pairs of agents. This weight with distance $e_{ij} = \norm{x_i-x_j}$ can also be interpreted in the MPP framework as a probability of success  under exponential decay with $p_{ij} = \exp(-\norm{x_i-x_j})$.

The graph $\gs$  is visualized in Figure \ref{fig:space_layout}, with the 10 vertices in $S$ marked as the large green dots. The maximum distance to $S$  within this graph is $d^*_{\max} = 844.370$, and the effective diameter of the graph with source set $S$ is $\mathcal{D}(\gs) = 15$.

\begin{figure}[!h]
        \centering
         \centering
    \begin{subfigure}[b]{0.45\linewidth}
    \includegraphics[width=\linewidth,height=3cm]{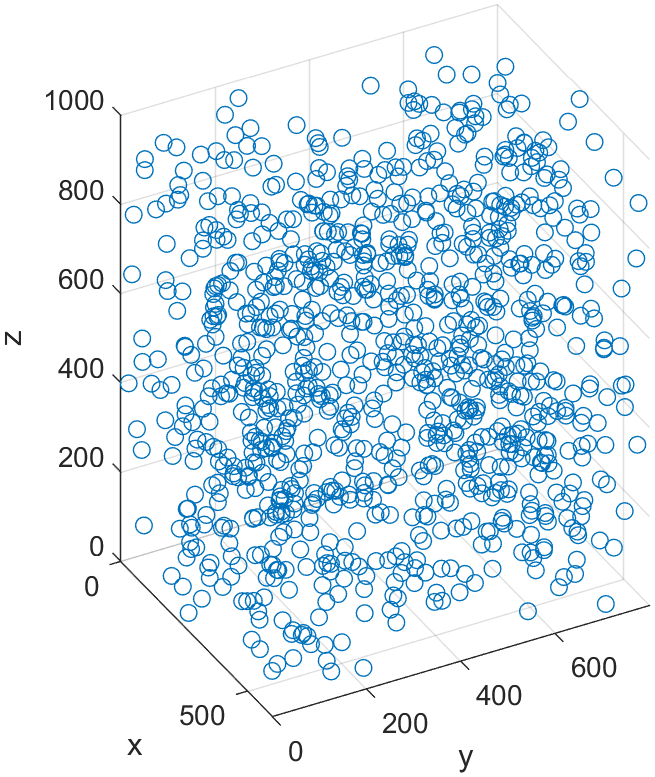}
        \caption{Position of the 1000 agents}
        \label{fig:space_position}
    \end{subfigure}\hfill 
        \begin{subfigure}[b]{0.45\linewidth}
\includegraphics[width=\linewidth,height=3cm]{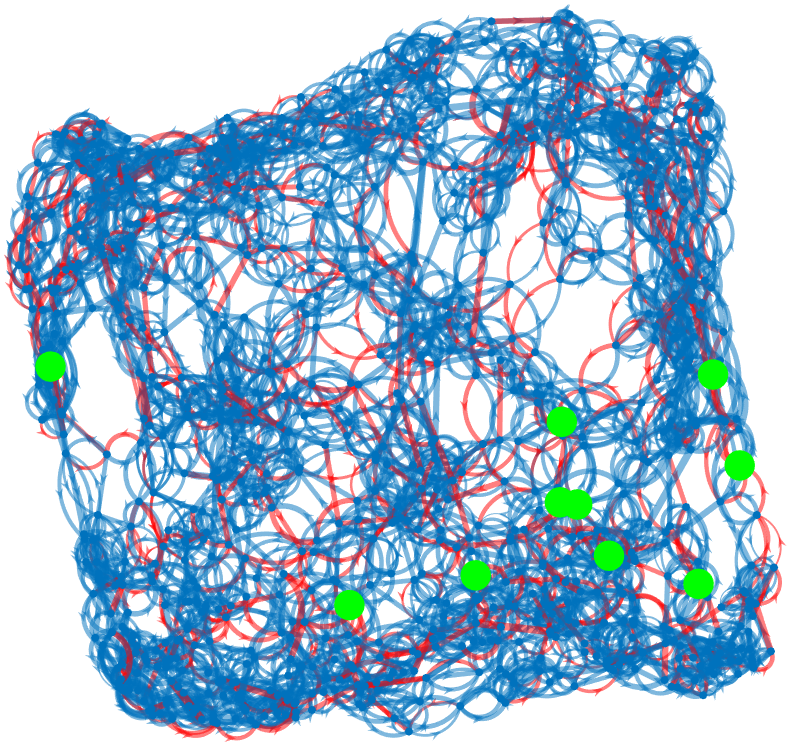}
        \caption{Graph topology of agent communication}
    \label{fig:space_layout}
    \end{subfigure}
    \caption{Position and communication structure for agents}
    \label{fig:space_background}
    \end{figure}

Figure \ref{fig:space_robustness} plots trajectories and robustness bounds with respect to Reads  $(P_\Rs = 8, \epsilon^{\Rs} \in [-1, 2])$, Updates $(P_\Us = 8, \epsilon^{\Us} \in [-3, 5])$, and Writes $(P_\Ws = 2, \epsilon^{\Ws} \in [-0.1, 0.1])$.  These trajectories arise from 60 initial conditions (with $d(0)$ uniformly sampled from $[0, 2000]^{\abs{\vs}+ \abs{\es}}$), for which distributed shortest-path computation is simulated 60 times for each initial condition (with respect to randomly generated $(\Rs(\cdot), \Us(\cdot), \Ws(\cdot))$ sequences and associated execution queues $\mathcal{Q}(\cdot)$). 
The worst-case times of $T^+ = 270$ and $T^-=4482$ to converge  to robustness levels $B^+ = 106.5$ and $B^-=61.5$ respectively are computed for this set of 60 initial conditions via \eqref{eq:robust_bound}. A log scale on the vertical axis is used in Figure \ref{fig:space_robustness} to visualize the fluctuating error levels.

    \begin{figure}[!h]
    \includegraphics[width=0.9\linewidth]{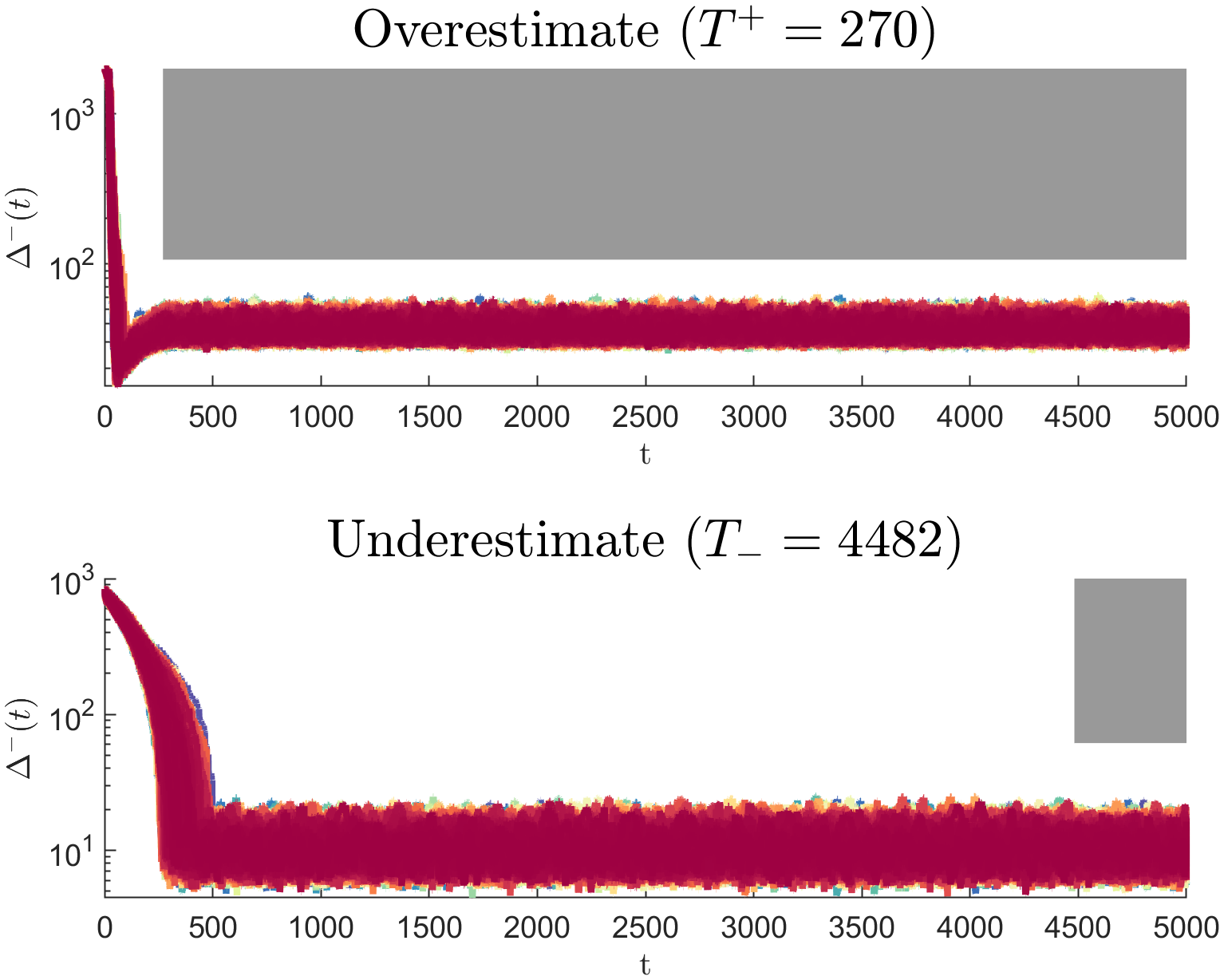}
    \caption{Robustness for the space-graph from Figure \ref{fig:space_background}}
    \label{fig:space_robustness}
    \end{figure}

\section{Conclusion}
\label{sec:conclusion}
We analyzed the performance and robustness of an asynchronous (and possibly noisy) adaptive Bellman--Ford algorithm for shortest path computation. Our asynchrony model  allows for  arbitrary execution orderings and race conditions in the computation. The convergence   and ultimate noise-robustness bounds we obtain  are simply $P$ times those obtained  for the synchronous scenario in \cite{mo2019robustness}, where $P$ is the asynchrony measure from \eqref{eq:p_time_window}. 

The bounds and gain values from this work can be used to perform compositional analyses of control algorithms integrated with distance estimates (e.g. path planning). 
For instance, the gain bounds can be used for upstream applications in flocking design and dispatch.

Open problems include finding expected times (and bounds on the worst-case expectation) for ultimate convergence in the stochastic setting, when the Update and Transmission times are chosen in some distribution (e.g., uniform in an interval of length $P$, or according to a Poisson counting process).

\section*{Acknowledgements}
The authors would like to thank Roy S. Smith, Jaap Eising,  Eduardo Sontag,  Corentin Briat, and the Systems Theory of Algorithms group at ETH Z\"{u}rich for their support and for discussions about asynchronous shortest path computation.

\bibliographystyle{IEEEtran}
\bibliography{references}

\newpage
\appendix
\section{Definitions}

\label{app:defn}

We use the following definitions in bounding the convergence time $T$ from Definition \ref{def:converge}.

\begin{defn}[Definition 1 of \cite{mo2019robustness}]
    A node $j$ is a \textbf{True Constraining Node} of $i$ ($j \in \text{TCN}(i)$) if $j \in \mathcal{N}(i)$ and $d_i^* = e_{ij} + d_j^*$.
\end{defn}

\begin{defn}[Definition 3 of \cite{mo2019robustness}]
    A node $j$ is a \textbf{Current Constraining Node} of $i$ ($j \in \text{CCN}(i; t$)) given $d_i(t)$ and $d_{ij}(t)$ if $j \in \mathcal{N}(i)$ and $d_i(t) = e_{ij} + d_{ij}(t)$.
\end{defn}

Note that the TCN and CCN sets are the same for each node if $d = d^*$.

\begin{defn}
    The \textbf{True (Current) Constraining Graph} is a 
    subgraph of $\mathcal{G}$ with an edge $(i, j)$ present only if $j \in \text{TCN}(i)$  ($j \in \text{CCN}(i; t)$). All edges of the Constraining Graph have weight 1 (i.e., it is unweighted).
\end{defn}
\begin{defn}
\label{defn:effective}
    The \textbf{Effective Diameter } $\mathcal{D}(\mathcal{G})$ is the maximum shortest-path length (measured in number of nodes) from any node to a source in $S$ in any weakly connected component of the True Constraining Graph of $\mathcal{G}$.
\end{defn}

The effective diameter $\mathcal{D}(\gs)$ depends on the choice of source nodes $S$. It is also upper-bounded by the Diameter of unweighted version of $\gs$ (the graph $\gs$ with all edge weights in $\gs$  replaced with $e_{ij} = 1$) plus one. The effective diameter may be interpreted as the maximal number of nodes crossed in any shortest path between a node in $\vs$ and $S$. Definition~\ref{defn:effective} is a \textit{directed} version of the definition of Effective Diameter in \cite[Definition 2]{mo2019robustness}. Please see  Figure~\ref{fig:Buckyball} for an example of effective diameter and true constraining graph. 

For the finite-time convergence of $\Delta^-(t)$, we need to define the set of underestimating elements
\begin{subequations}\allowdisplaybreaks
\label{eq:underestimators}
\begin{align}\allowdisplaybreaks
    \s^-_\Us(t) &= \{i \in \vs \mid \ \Delta_i(t) < 
    0\}, \\
    \s^-_\Ws(t) &= \{(i, j) \in \es \mid \ \Delta_{ij}^m(t) < 
    0\}, \\
    \s^-_\Rs(t) &= \{(i, j) \in \es \mid \ \Delta_{ij}(t) < 
    0\} \\
    \s(t) &= \s^-_\Us(t) \cup \s^-_\Ws(t) \cup \s^-_\Rs(t) \\
    \intertext{and the minimal-value underestimates  in all buffers $(D^{\min}_\Us(t), D^{\min}_\Ws(t),D^{\min}_\Rs(t))$}
    D^{\min}_\Us(t) &= \begin{cases}
        \min_{i \in \s^-_\Us(t)} d_i(t) & \s^-_\Us(t) \neq \varnothing \\
        \infty & \s^-_\Us(t) = \varnothing
    \end{cases} \\
    D^{\min}_\Ws(t) &= \begin{cases}
        \min_{(i, j)  \in \s^-_\Ws(t)} m_{ij}(t) & \s^-_{\Ws}(t) \neq \varnothing \\
        \infty & \s^-_\Ws(t) = \varnothing
    \end{cases} \\
    D^{\min}_\Rs(t) &= \begin{cases}
        \min_{(i, j) \in \s^-_\Rs(t)} d_{ij}(t) &   \s^-_\Rs(t) \neq \varnothing \\
        \infty & \s^-_\Rs(t) = \varnothing 
    \end{cases}. \label{eq:dmin_Read}
\end{align} 
\end{subequations}
The minimum-value distance underestimate  is
\begin{align}\label{eq:dmin}
    D^{\min}(t) = \min\left[ D^{\min}_\Us(t), D^{\min}_\Ws(t), D^{\min}_\Rs(t)\right].
\end{align}
Equations \eqref{eq:underestimators} and \eqref{eq:dmin} are based on Eq. (35) and (36) from \cite{mo2019robustness} in the synchronous formulation.

\section{Proof of Lemma \ref{lem:monotonic_alt}: Non-increase of $\Delta^+, \Delta^-$}
\label{app:monotonic_alt}

We prove that the Read, Update, and Write operations are monotone on the errors $\Delta^+$ and $\Delta^-$. As a consequence,   the sequences $\{\Delta^+(t)\}$ and $\{\Delta^-(t)\}$ (generated by sequences of Reads, Updates, and Writes) are monotonically non-increasing.  For the Update step, the  arguments are adapted from Lemmas 2 and 3 of \cite{mo2019robustness}. 

Let $\tm $ and $\tp $ be instances before and after a Read or Write on $(i, j)$ or  an Update or $i$ (according to the computation model, in case of race conditions, we could have $\tm = t-k\delta$ and $\tp = t-(k-1)\delta$ as appropriate).
The goal is to show $ \Delta^+(\tp) \leq \Delta^+(\tm)$ and $\Delta^-(\tp) \leq  \Delta^-(\tm)$.

For an Update operation, since Update$(i)$ only changes the value of $d_i$ , we have
\begin{align*}
    \text{Update}(i): \quad  \Delta^+ (\tp) & \leq \max \{ \Delta^+(\tm), \Delta_{i}(\tp)\} &     
    \\
      \Delta^- (\tp) & \leq \max \{ \Delta^-(\tm), -\Delta_{i}(\tp)\}
      \end{align*}
Therefore, we only need to prove that 
\begin{subequations}
\label{eq:goal_u}
\begin{align}
\label{eq:goal1_u}
    \Delta_{i}(\tp) & \leq \Delta^+ (\tm) 
    \\
    \label{eq:goal2_u}
     -\Delta_{i}(\tp) & \leq \Delta^- (\tm)
\end{align}
\end{subequations}
If $i \in S$, then  $\Delta_{i}(\tp)=0$ and \eqref{eq:goal_u} holds. 
If  $i \not\in S$, by the Update rule in \eqref{eq:update} we have: 
for some $j \in \text{TCN}(i)$, 
\begin{subequations}
    \label{eq:monotone_over}
\begin{align}
     \Delta_i(\tp) & \leq d_{ij}(\tm)+e_{ij} -d_i^* 
    \\&= d_{i j}(\tm) + e_{i j} - (e_{i j} +  d^*_j) \\
        &= d_{i j}(\tm) - d^*_j =  \Delta_{i j}(\tm)     \leq \Delta^+(\tm), 
\end{align}
\end{subequations}
which implies that \eqref{eq:goal1_u} holds. For some $j \in \text{CCN}(i,\tp)$
\begin{subequations}
    \label{eq:monotone_under}
\begin{align}
    -\Delta_{i}(\tp) & =  d^*_i - (d_{ij}(\tm)+e_{ij})  \\ 
         & \leq  d^*_j + e_{i j}  - (d_{i j}(\tm) + e_{i j}) \\
         &= d^*_j - d_{i j}(\tm) = -\Delta_{i j}(\tm)   \leq \Delta^-(\tm). 
\end{align}
\end{subequations}
meaning that \eqref{eq:goal2_u} holds.
For a Read operation, Read$(i,j)$ only changes the value of $d_{ij}$. Thus, 
\begin{align}
    \Delta^+ (\tp) & \leq \max \{ \Delta^+(\tm), \Delta_{ij}(\tp)\} \\
      \Delta^- (\tp) & \leq \max \{ \Delta^-(\tm), -\Delta_{ij}(\tp)\}
\end{align}
and we only need to prove
\begin{subequations}\label{eq:goal}
    \begin{align}
 \Delta_{ij}(\tp) & \leq \Delta^+ (\tm), & 
    \label{eq:goal2}
     -\Delta_{ij}(\tp) & \leq \Delta^- (\tm).
\end{align}  
\end{subequations}
However,   $\Delta_{ij}(\tp) = d_{ij}(\tp)-d_j^\star = m_{ij}(\tm)-d_j^\star $, and thus \eqref{eq:goal} holds.   Similarly for a Write operation, we have 
\begin{align}
    \Delta^+ (\tp) & \leq \max \{ \Delta^+(\tm), \Delta_{i}^m (\tp)\} \\
      \Delta^- (\tp) & \leq \max \{ \Delta^-(\tm), -\Delta_{i}^m(\tp)\}
\end{align}
and it suffices to prove
\begin{subequations}\label{eq:goal_write}
    \begin{align}
 \Delta_{i}^m(\tp) & \leq \Delta^+ (\tm), \label{eq:goal1_write} & 
     -\Delta_{i}^m(\tp) & \leq \Delta^- (\tm).
\end{align}  
\end{subequations}
Since $\Delta_{i}^m(\tp) = m_{ij}(\tp)-d_j^\star = d_j(\tm)-d_j^\star$,  \eqref{eq:goal_write} holds.

\section{Proof of Th.~\ref{thm:ultimate_deltaplus}: Finite-Time Convergence of $\Delta^+$}
\label{app:ultimate_deltaplus}
 Let $n_1, n_2, \ldots n_T \in \vs$ be a sequence of agents such that $n_1 \in S$ and $\forall k \in \{1,...,T-1\}$,  $ n_k \in \text{TCN}(n_{k+1})$. The length of such a sequence is upper-bounded as $T \leq \Dc(\gs)$ because of the definition of effective diameter (Definition \ref{defn:effective}) as the maximum length of a True Constraining Node sequence in any connected component of the True Constrained Graph.
The following assertion will be proven by induction: $\forall k \in \{1,...,T\}$,
\begin{subequations}\label{eq:induction}
\begin{align}
    \Delta_{n_k}(t) \leq 0 & & & \forall t \geq P k - P_{\Rs} - P_{\Ws}\label{eq:induction_hypo_update} \\
      \Delta_{n_{k+1},n_{k}}^m(t) \leq 0 & & & \forall t \geq P k - P_{\Rs}\label{eq:induction_hypo_m} \\
      \Delta_{n_{k+1},n_k}(t) \leq 0 & & & \forall t \geq P k.\label{eq:induction_hypo_r}
\end{align}
\end{subequations}
Note that this immediately implies the Theorem, since every vertex $i$ and every edge $(i,j)$ is part of a 
True constrained Node Sequence of length $T\leq \mathcal D(\mathcal G)$, by A1.
Furthermore, the following  arguments still holds if $P_\Ws =0$ or $P_\Rs=0$, due to the ordering convention.  

For the induction base, we prove \eqref{eq:induction} for $k =1$. Given a source node $n_1 \in S$, by Assumption A2 there exists 
$t_1 \in \{1,\dots,P_\Us\}$ such that $n_1 \in \mathcal{U}(t_1)$. Hence, by the Update \eqref{eq:update}, it holds that $\Delta_{n_1}(t) = \Delta_{n_1}(P_\Us)=0$ for all $t \geq t_1$, which proves \eqref{eq:induction_hypo_update}. Again, by Assumption A2, there is $t_2 \in \{P_\Us\rw{+1},\dots,P_\Us+P_\Ws\}$ such that $(n_{2},n_1) \in \Ws(t_2)$; hence, $\Delta_{n_2,n_1}^m(t)=0$ for all $t\geq t_2$, which proves \eqref{eq:induction_hypo_m}. Finally, by Assumption A2, there is $t_3 \in \{P_\Us+P_\Ws\rw{+1}, \dots, P_\Us+P_\Ws+P_\Rs\}$ such that $(n_{2},n_1) \in \Rs(t_3)$; hence $\Delta_{n_2,n_1}(t) = 0$ for all $t \geq t_3$, which implies \eqref{eq:induction_hypo_r}. 

We now focus on the induction step: assuming \eqref{eq:induction} holds for $k$, we prove it for $ k+1$. 
First, by A2, there is $t_1 \in \{Pk+1,\dots,Pk+P_\Us\}$ such that agent $n_{k+1}$ performs an Update, i.e., $n_{k+1} \in \Us(t_1)$. Let us call $d_{n_{k+1},n_k}(t_1^-)$ the value of $d_{n_{k+1},n_k}$ used in this Update, where $t_1^- $ can be equal to $t_1-1$ or $t_1$ (the latter if Read$(n_{k+1},n_k)$ is performed at time $t_1$ preceding the Update in the instruction queue). We have  
\begin{align*}
    \Delta_{n_{k+1}}(t_1)&  =  d_{n_{k+1}}(t_1) - d_{n_{k+1}}^*
    \\
    & \overset{(i)}{\leq}
     d_{n_{k+1},n_k}(t_1^-)+e_{n_{k+1},n_k}-d_{n_{k+1}}^*
    \\ 
     & \overset{(ii)}{=} d_{n_{k+1},n_k}(t_1^-) +e_{n_{k+1},n_k}-(e_{n_{k+1},n_k} +d_{n_{k}}^*)
    \\ 
     & =\Delta_{n_{k+1},n_k}(t_1^-) \overset{(iii)}{\leq} 0
\end{align*}
where (i) follows by the Update rule \eqref{eq:update}, (ii) because $n_{k}\in\text{TCN}(n_{k+1})$, and (iii) is the induction hypothesis \eqref{eq:induction_hypo_r}. 
Since the same argument holds for any further update of $d_{n_{k+1}}$ after $t_1$, we conclude \begin{align*}
    \Delta_{n_{k+1}}(t) \leq 0, \quad \forall t\geq Pk +P_\Us= P(k+1)-P_\Rs-P_\Ws,
\end{align*} which is the first part of the induction step \eqref{eq:induction_hypo_update}. The other claims \eqref{eq:induction_hypo_m},  \eqref{eq:induction_hypo_r} follow similarly, as in the induction base.

\section{Proof of Th.~\ref{thm:ultimate_deltaminus}: Finite-Time Convergence of $\Delta^-$}
\label{app:ultimate_deltaminus}

The proof is based on the following Lemma. 

\begin{lem}\label{lem:minimumincrease}
The minimum-value underestimate $D^{\min}$ in \eqref{eq:dmin} satisfies a bounded-increase condition: for all $t\in \mathbb N$
\begin{align}
    D^{\min}(t+P) \geq D^{\min}(t) + e_{\min}.\label{eq:dmin_kncrease}
\end{align}

\end{lem}
\begin{proof}
    Let $\tm$ and $\tp$ the instants before and after an action (Read, Write, or Update) as in Appendix \ref{app:monotonic_alt}.
We start by showing that $D^{\min}$ is nondecreasing, namely that 
\begin{align}\label{eq:goal3}
    D^{\min}(\tp) &\geq  D^{\min}(\tm).
\end{align}

\noindent We consider four cases, each of which implies  \eqref{eq:goal3}:
\begin{enumerate}
    \item The action is Update$(i)$ and $i \notin \s^{-}_\Us(\tp)$; or the action is 
    Read$(i,j)$ and $(i,j) \notin \s^{-}_\Rs(\tp)$; or the action is Write$(i,j)$ and $(i, j) \notin \s^{-}_\Ws(\tp)$. 
    \item The action is Read$(i,j)$  and $(i, j) \in \s^{-}_\Rs(\tp)$. Then  $j \in \s^{-}_\Us(\tm)$  and 
    $D^{\min}(\tp) = {\min} \{ D^{\min}(\tm),d_{ij}(\tp)\} = {\min } \{D^{\min}(\tm),m_{ij} (\tm) \} = D^{\min}(\tm)$,
   \item The action is Write$(i,j)$ and $(i,j) \in \mathcal{S}^-_{\Ws}(t^+)$. Then $j \in \mathcal{S}^-_{\Us}$  and $D^{\min}(\tp) = {\min} \{ D^{\min}(\tm),m_{ij}(\tp)\} = {\min } \{D^{\min}(\tm),d_{j} (\tm) \} = D^{\min}(\tm)$.
    \item The action is Update and $i \in \s^{-}_\Us(\tp)$, then $ d_{i}(\tp) = d_{ik}(\tm)+e_{ik}$ for some $(i,k) \in \s^{-}_\Ws(\tm)$, and thus 
    $D^{\min}(\tp) = {\min} \{ D^{\min}(\tm),d_{i}(\tp)\} = {\min } \{D^{\min}(\tm),d_{ik}(\tm)+e_{ik} \} \geq  D^{\min}(\tm)$.
\end{enumerate}

Hence, \eqref{eq:goal3} holds. Let now $\bar t \in \mathbb N$ be arbitrary. Since all the variables are updated via a sequence of Updates, Writes, Reads, by \eqref{eq:goal3}, we have $D^{\min}(t) \geq  D^{\min}(\bar t)$ for all $t \geq \bar t$. Let us show that this implies \eqref{eq:dmin_kncrease}.

First, if $i$ performs an Update at time $t_1 > \bar t$, and after the Update $i \in \s^{-}(\tp_1)$, then  it necessarily holds that $ d_{i}(\tp_1) = d_{ik}(\tm_1)+e_{ik}$ for some $(i,k) \in \s^{-}_\Ws(\tm_1)$. Therefore, 
$
    d_i(\tp_1)  \geq D^{\min}(\tm_1) + e_{ik}
             \geq D^{\min}(\bar t) + e_{\min}.
$
In particular,
\begin{align}\label{eq:boundedincrease_u}
   d_i(t) \geq D^{\min}(\bar t) + e_{\min}, \quad \forall t\geq \bar t + P_\Us,
\end{align}
for all $i \in \s^{-}_\Us(t)$ (since every agent $i$ performed at least one Update after $\bar t$, and the only way for $i$ to ``enter'' the set $\s^{-}_\Us(t)$ is by performing an update). 

Second, between time $\bar t + P_\Us$ and $\bar t + P_\Us + P_\Ws$, each $m_{ij}$ is Written at least once after time $\bar t + P_\Us$, say at time $t_2$. Then $(i, j) \in \s^{-}_\Ws(t_2^+)$ if and only if $j \in   \s^{-}_\Us(t_2^-)$. By \eqref{eq:boundedincrease_u}, we can conclude that 
\begin{align}\label{eq:boundedincreasem}
    m_{i,j}(t) \geq D^{\min}(\bar t) + e_{\min}, \quad \forall t\geq \bar t + P_\Us+P_\Ws,
\end{align}
for all $(i,j)\in \s^{-}_\Ws(t)$. 

Third,  between time $\bar t + P_\Us+P_\Ws$ and $\bar t + P_\Us + P_\Ws+ P_\Rs$, each $d_{ij}$ is Read at least once after time $\bar t + P_\Us+P_\Ws$, say at time $t_3$. Then $(i, j) \in \s^{-}_\Rs(t_3^+)$ if and only if $(i,j) \in   \s^{-}_\Rs(t_3^-)$. By \eqref{eq:boundedincreasem}, we can conclude that 
\begin{align}\label{eq:boundedincreased}
    d_{i,j}(t) \geq D^{\min}(\bar t) + e_{\min}, \quad \forall t\geq \bar t + P_\Us+P_\Ws +P_\Rs,
\end{align}
for all $(i,j)\in \s^{-}_\Rs(t)$. Putting together \eqref{eq:boundedincrease_u}, \eqref{eq:boundedincreasem},  \eqref{eq:boundedincreased}, and the definition of $D^{\min}$, we obtain \eqref{eq:dmin_kncrease}. 
\end{proof}

We can now resume the task of proving Theorem \ref{thm:ultimate_deltaminus}. By the definition in \eqref{eq:dmin}, it must hold that either $D^{\min} < d^*_{\max}$ (since $D^{\min}$ is the minimum of the \emph{underestimates}), or $D^{\min} = \infty$ (the latter only in case $\mathcal S(t) =\varnothing$, i.e., there is no underestimate in the system). However, the previous Lemma states that   $D^{\min}$ will  increase at least by $e_{\min}$ every $P$ time instants:
$
    D^{\min}(t) \geq D^{\min}(0) + \lfloor{t/P} \rfloor e_{\min}.$
Therefore, if
\begin{align}
  t \geq T_P^{-} \coloneqq   P \left \lceil \frac{d^*_{\max} - D^{\min}(0)}{e_{\min}} \right \rceil,
\end{align}
then we are guaranteed that $D^{\min}(t) \geq d_{\max}^*$, i.e., $D^{\min} =\infty$. This means that  $\Delta^{-}(t) = 0$ for all $t \geq T_P^{-}$.

\section{Proof of  Theorem~\ref{thm:robust}: Robustness Bounds}

\label{app:lem_robustness}
The robustness bounds in Theorem \ref{thm:robust} will be proven using a comparison lemma:

\begin{lem}
\label{lem:comparison}
    Let $D(\cdot), D^+(\cdot), D^-(\cdot)$ be distance estimates sequences generated by \eqref{eq:noisydynamics} and \eqref{eq:minandmaxupdates}, with the same initial condition $D(0) = D^+(0) = D^{-}(0)$, and the same arbitrary event sequences $\Rs(\cdot), \Us(\cdot), \Ws(\cdot)$ and choice of ordering between the respective actions. The following hold for all $t \in \mathbb N$: \vspace{-1.5em}
    \begin{subequations}        
    \begin{align}
        \forall i \in \vs: \quad  & d_i^-(t) \leq d_i(t) \leq d^+_i(t) \\
        \forall (i,j) \in \es: \quad  & m_{ij}^-(t) \leq m_{ij}(t) \leq m^+_{ij}(t) \\
        \forall (i,j) \in \es: \quad & d_{ij}^-(t) \leq d_{ij}(t) \leq d^+_{ij}(t).
    \end{align}
    \end{subequations}
\end{lem}
\begin{proof}
    As in the previous proofs we denote with $t^-$ and $t$ the time index for the state just before and just after an action, respectively. We prove \rw{by induction} that the order  $D^- \leq D \leq D^+$ is preserved under Read, Write, and Update actions: in particular, we show that $D^-(t^-) \leq D(t^-) \leq D^+(t^-)$ implies that 
    $D^-(t) \leq D(t) \leq D^+(t)$. 

    \emph{Read:} If the action is Read$(i, j)$, then 
    \begin{align*}
            d^-_{ij}(t) &= m_{ij}^-(t^-)+ \epsilon_{\min}^\Rs \\
            &\leq m_{ij}(t^-)+ \epsilon_{ij}^\Rs = d_{ij}(t) \\
             & \leq  m^+_{ij}(t^-)+ \epsilon_{\max}^\Rs = d_{ij}^+(t).
        \end{align*}

    \emph{Update:} Let the action be Update$(i)$. If $i \in S$, $d^-_{i}(t) = d_{i}(t) = d^+_{i}(t) = 0$. Otherwise, 
    \begin{align*}
        d_{i}^-(t) &= \min_{j \in \Ns(i)} d_{ij}^-(t^-) + (e_{ij} +\epsilon_{\min}^\Us)
        \\
        & \leq \min_{j \in \Ns(i)} d_{ij}(t^-) + e_{ij} +\epsilon_{ij}^\Us = d_i(t) 
        \\  
        & \leq \min_{j \in \Ns(i)} d_{ij}^+(t) + e_{ij} +\epsilon_{\max}^\Us = d_i^+(t).
    \end{align*}

    \emph{Write:} If the action is Write$(i,j)$, then 
        \begin{align*}
            m^-_{ij}(t) &= d_j^-(t^{-})+ \epsilon_{\min}^\Ws \\
            &\leq d_j(t^{-})+ \epsilon_{ij}^\Ws = m_{ij}(t) \\
            & \leq  d^+_j(t^{-})+ \epsilon_{\max}^\Ws = m_{ij}^+(t). \nonumber
            \\[-3em]
        \end{align*}
\end{proof}

Note that \eqref{eq:minandmaxupdates} corresponds to the (noiseless) ABF algorithm applied on the graphs $\gs^+$ and $\gs^-$, respectively. Hence, by Theorem~\ref{thm:ultimate_deltaplus} we have that $D^+$ will have zero overestimate error within $P\mathcal{D}   (\gs^+)$ (with respect to the shortest path solution computed on $\gs^+$);  and $D^-$ will have zero underestimate error in time $P \left \lceil \frac{d^*_{\max}(\gs^-) - D^{\min}(0)}{e_{\min}-\epsilon_{\min}} \right \rceil$  by Theorem \ref{thm:ultimate_deltaminus}. Hence, we only need to provide bounds for the difference between the shortest path distances $D^{+*}$, $D^{-*}$ computed over $\gs^+$, $\gs^-$ and the shortest path distance $D^*$. 

\begin{lem}
\label{lem:fixed_pos}
    The fixed point of the sequence $D^+(t)$ obeys   $\norm{D^{+*} - D^{*}}_{\infty} \leq \epsilon_{\max}(\mathcal{D}(\gs) - 1)$.
\end{lem}
\begin{proof}
    By definition of effective diameter, for any $i\in 
    \vs$, there is a shortest path $\mathcal P$ from $i$ to $S$ (i.e., a path in the $\argmin$ of \eqref{eq:shortestdistance}) composed of not more that $\mathcal D(\gs)-1$ edges. By definition of $\gs^+$, the length of the same path $\mathcal P$ from $i$ to $S$ over $\gs^+$ is at most equal to the length of $\mathcal P$ over $\gs$ plus $\mathcal (D(\gs)-1) \epsilon_{\max}$. Noting that the length of $\mathcal P$ over $\gs$ is not larger than $d^*_{\max}$ gives the conclusion. 
\end{proof}

\begin{cor}
\label{cor:fixed_minus}
    The fixed point of the sequence 
    $D^-(t)$ satisfies $\norm{d^{-*} - d^*}_\infty \leq \epsilon_{\min}(\mathcal{D}(\gs^-)-1).$
\end{cor}
\begin{proof}
The sequence $D^-(t)$ may be interpreted as a shortest-path computation occurring over the graph $\gs^-$ w.r.t. nonnegative additive noise $\epsilon \in [0, \epsilon_{\min}]$. The objective bound of $\epsilon_{\min}(\mathcal{D}(\gs^-) -1)$ therefore holds by Lemma \ref{lem:fixed_pos}.
\end{proof}

The robustness bounds of Theorem \ref{thm:robust} are therefore proven by Lemma \ref{lem:fixed_pos} and Corollary \ref{cor:fixed_minus}.

\end{document}